\documentclass[12pt, a4paper, reqno]{amsart}
\usepackage{amsmath}
\usepackage{amsfonts}
\usepackage{amssymb}
\usepackage{color}
\usepackage{comment,graphicx}
\usepackage[T1]{fontenc}
\usepackage{url}
\usepackage{ae}

\def\phi{\varphi}
\def\rho{\varrho}
\def\epsilon{\varepsilon}
\numberwithin{equation}{section}
\theoremstyle{plain}
\newtheorem{theorem}[equation]{Theorem}
\newtheorem{lemma}[equation]{Lemma}

\theoremstyle{definition}
\newtheorem{definition}[equation]{Definition}

\theoremstyle{remark}
\newtheorem{remark}[equation]{Remark}

\renewcommand{\le}{\leqslant}
\renewcommand{\ge}{\geqslant}
\renewcommand{\leq}{\leqslant}
\renewcommand{\geq}{\geqslant}

\begin{document}
\title[Hardy--H\'{e}non parabolic equations in Herz spaces]{Well-posedness and
global existence for Hardy--H\'{e}non parabolic equations in Herz spaces}
\author{Douadi Drihem}
\address{Douadi Drihem\\
	Laboratory of Functional Analysis and Geometry of Spaces, Faculty of
	Mathematics and Informatics, Department of Mathematics,\\
	M'sila University, PO Box 166 Ichebelia, M'sila 28000, Algeria}
\email{douadidr@yahoo.fr, douadi.drihem@univ-msila.dz}
\thanks{ }
\date{\today }
\maketitle

\begin{abstract}
	In this paper, we study the Hardy--H\'{e}non parabolic equation
	\begin{equation*}
	\partial_t u=\Delta u+a|x|^{-\gamma}|u|^{\beta}u,
	\qquad
	t>0,\; x\in\mathbb{R}^{n}\setminus\{0\},
	\quad
	\beta>0,\; a\in\mathbb{R},
	\end{equation*}
	under suitable assumptions on the parameter $\gamma$. We establish the local and global well-posedness of mild solutions in homogeneous Herz spaces.


\end{abstract}

\section{Introduction}

The Hardy--H\'{e}non parabolic equation is a semilinear heat equation with a
spatially weighted nonlinearity. A standard form is 
\begin{equation}
\partial _{t}u=\Delta u+a|x|^{-\gamma }|u|^{\beta }u,\qquad t>0,\;x\in 
\mathbb{R}^{n}\setminus \{0\},\quad \beta >0,\;\gamma ,a\in \mathbb{R},
\label{HH-equation}
\end{equation}%
with initial value $u(0)=u_{0}$. The exponent $\gamma $ determines the
nature of the weight:

\begin{itemize}
\item \textbf{Hardy case:} $\gamma>0$, for which the weight $|x|^{-\gamma}$
is singular at the origin.

\item \textbf{H\'{e}non case:} $\gamma<0$, corresponding to the weight $%
|x|^{|\gamma|}$, which vanishes at the origin and grows at infinity.

\item \textbf{Fujita case:} $\gamma=0$, recovering the classical semilinear
heat equation 
\begin{equation*}
\partial_tu=\Delta u+a|u|^{\beta}u.
\end{equation*}
\end{itemize}

Hardy--H\'{e}non parabolic equations have been extensively studied from
various perspectives, including local and global well-posedness, critical
spaces and scaling invariance, existence and uniqueness of mild solutions,
blow-up versus global existence, Fujita-type critical exponents,
self-similar and asymptotically self-similar solutions, as well as decay
estimates and long-time asymptotic behavior. These problems have been
investigated in several functional settings, including weighted Lebesgue and
weighted Lorentz spaces with power weights; see \cite{BTW17,CIT22,CITT26}
and the references therein.

Herz spaces, introduced in \cite{Herz68}, play an important role in harmonic
analysis, particularly in the study of singular integral operators and
Fourier multipliers. Their importance stems not only from their theoretical
interest but also from their wide range of applications in analysis; see,
for example, \cite{BS85,Dr-Herz-Heat,Fi-We08,Rag09,Tu11}. For recent
developments and further results on Herz spaces, we refer the reader to \cite%
{RS,ZYZ} and the monograph \cite{LYH}.

In this paper, we study the Hardy--H\'{e}non parabolic equation %
\eqref{HH-equation} in Herz spaces. The paper is organized as follows. In
Section~2, we introduce the necessary notation and preliminaries on Herz
spaces, including interpolation inequalities and elementary embedding
results. We also establish the heat semigroup estimates in Herz spaces that
are required to handle the singular weighted nonlinearity in %
\eqref{HH-equation}.

In Section~3, we prove local well-posedness results for \eqref{HH-equation}
in Herz spaces. Finally, in Section~4, we investigate the global existence
of solutions under suitable assumptions on the initial data.

The results established in this paper can be naturally extended to
Lorentz--Herz spaces by replacing the Lebesgue spaces in the definition of
Herz spaces with the corresponding Lorentz spaces. However, to avoid
unnecessary technical complications and to keep the presentation as clear as
possible, we restrict our attention to the Herz space setting.

\section{Herz spaces}

In this section, we present some fundamental properties of Herz spaces. As
usual, $\mathbb{R}^{n}$ denotes the $n$-dimensional Euclidean space, $%
\mathbb{N}$ the set of all natural numbers, and $\mathbb{N}_{0}=\mathbb{N}%
\cup \{0\}$. The symbol $\mathbb{Z}$ stands for the set of all integers. For 
$x\in \mathbb{R}^{n}$ and $r>0$, we denote by $B(x,r)$ the open ball in $%
\mathbb{R}^{n}$ centered at $x$ with radius $r$. If $1\leq p<\infty $ and $%
\frac{1}{p}+\frac{1}{p^{\prime }}=1,$ then $p^{\prime }$ is called the
conjugate exponent of $p$.

\medskip

The Lebesgue space $L^{p}(\mathbb{R}^{n})$, $0<p\leq \infty $, consists of
all measurable functions $f$ such that 
\begin{equation*}
\Vert f\Vert _{L^{p}(\mathbb{R}^{n})}=\left( \int_{\mathbb{R}%
^{n}}|f(x)|^{p}\,dx\right) ^{1/p}<\infty ,\qquad 0<p<\infty ,
\end{equation*}%
and 
\begin{equation*}
\big\Vert f\big\Vert_{L^{\infty }(\mathbb{R}^{n})}=\underset{x\in \Omega }{%
\text{ ess-sup}}\left\vert f(x)\right\vert <\infty .
\end{equation*}%
We simply write 
\begin{equation*}
\Vert f\Vert _{L^{p}(\mathbb{R}^{n})}=\Vert f\Vert _{p}.
\end{equation*}

Let $k\in \mathbb{Z}$. For convenience, we set 
\begin{equation*}
B_{k}=B(0,2^{k}),\qquad R_{k}=B_{k}\setminus B_{k-1}.
\end{equation*}%
We denote by $\chi _{k}$ the characteristic function of the set $R_{k}$.

\begin{definition}
\label{def-inh-Herz} Let $\alpha \in \mathbb{R}$ and $1\leq p,q\leq \infty $%
. The homogeneous Herz space $\dot{K}_{p}^{\alpha ,q}$ is defined as the set
of all functions $f\in L_{\mathrm{loc}}^{p}(\mathbb{R}^{n}\setminus \{0\})$
such that 
\begin{equation*}
\Vert f\Vert _{\dot{K}_{p}^{\alpha ,q}}=\left( \sum_{k=-\infty }^{\infty
}2^{k\alpha q}\Vert f\chi _{k}\Vert _{p}^{q}\right) ^{1/q}<\infty ,
\end{equation*}%
with the usual modifications when $p=\infty $ and/or $q=\infty $.
\end{definition}

\begin{remark}
The spaces $\dot{K}_{p}^{\alpha ,q}$ are Banach spaces. If $\alpha =0$ and $%
1\leq p=q\leq \infty $ , then $\dot{K}_{p}^{0,p}$ coincide with the Lebesgue
space $L^{p}(\mathbb{R}^{n})$. Moreover, 
\begin{equation*}
\dot{K}_{p}^{\alpha ,p}=L^{p}(\mathbb{R}^{n},|\cdot |^{\alpha p}),
\end{equation*}%
that is, the weighted Lebesgue space endowed with the norm 
\begin{equation*}
\Vert f\Vert _{L^{p}(\mathbb{R}^{n},|\cdot |^{\alpha p})}=\left( \int_{%
\mathbb{R}^{n}}|f(x)|^{p}|x|^{\alpha p}\,dx\right) ^{1/p}.
\end{equation*}

Let $0<p\leq \infty $, $0<q_{1}\leq q_{2}\leq \infty $, and $\alpha \in 
\mathbb{R}$. Then the following continuous embeddings hold: 
\begin{equation}
\dot{K}_{p}^{\alpha ,q_{1}}\hookrightarrow \dot{K}_{p}^{\alpha ,q_{2}}.
\label{herz-emb}
\end{equation}
\end{remark}

A detailed discussion of the properties of these spaces may be found in the
monograph \cite{LYH08}, the papers \cite{LuYang1.95}, \cite{LuYang2.95}, 
\cite{Hernandez1998}, and the references therein.

\begin{remark}
Let $V_{\alpha ,p,q}$ denote the set of all $(\alpha ,p,q)\in \mathbb{R}%
\times \lbrack 1,\infty ]^{2}$ such that:\newline
1. $\alpha <n-\dfrac{n}{p}$, $1\leq p\leq \infty $, and $1\leq q\leq \infty $%
;\newline
2. $\alpha =n-\dfrac{n}{p}$, $1\leq p\leq \infty $, and $q=1$. \newline
Let $1\leq p,q\leq \infty $ and $\alpha \in \mathbb{R}$. Then%
\begin{equation*}
\langle T_{f},\varphi \rangle =\int_{\mathbb{R}^{n}}f(x)\varphi
(x)\,dx,\qquad \varphi \in \mathcal{D}(\mathbb{R}^{n}),
\end{equation*}%
defines a regular distribution $T_{f}\in \mathcal{D}^{\prime }(\mathbb{R}%
^{n})$ for every $f\in \dot{K}_{p}^{\alpha ,q}$ if and only if $(\alpha
,p,q)\in V_{\alpha ,p,q}$. For the proof; see \cite{Dr-Sobolev}.
\end{remark}

We now present an interpolation inequality for Herz spaces.

\begin{lemma}
\label{interpolation2} Let $0<p_{0},p_{1},q_{0},q_{1}\leq \infty $, and $%
\alpha _{0},\alpha _{1}\in \mathbb{R}$. Define 
\begin{equation*}
\alpha =(1-\theta )\alpha _{0}+\theta \alpha _{1},\qquad \frac{1}{p}=\frac{%
1-\theta }{p_{0}}+\frac{\theta }{p_{1}},\qquad \frac{1}{q}=\frac{1-\theta }{%
q_{0}}+\frac{\theta }{q_{1}}.
\end{equation*}%
Then the interpolation inequalities 
\begin{equation}
\Vert f\Vert _{\dot{K}_{p}^{\alpha ,q}}\leq \Vert f\Vert _{\dot{K}%
_{p_{0}}^{\alpha _{0},q_{0}}}^{1-\theta }\Vert f\Vert _{\dot{K}%
_{p_{1}}^{\alpha _{1},q_{1}}}^{\theta }  \label{Interpolation}
\end{equation}%
hold for all $f\in \dot{K}_{p_{0}}^{\alpha _{0},q_{0}}\cap \dot{K}%
_{p_{1}}^{\alpha _{1},q_{1}}$.
\end{lemma}

In what follows we use the following lemma.

\begin{lemma}
\label{Maximal-Inq}Let\ $m>n,R>0,1\leq p\leq \infty $\ and\ $0<q\leq \infty $%
. Let $f\in \dot{K}_{p}^{\alpha ,q}$ and $-\frac{n}{p}<\alpha <n(1-\frac{1}{p%
})$. Then 
\begin{equation*}
\big\|\eta _{R,m}\ast f\big\|_{\dot{K}_{p}^{\alpha ,q}}\lesssim \big\|f\big\|%
_{\dot{K}_{p}^{\alpha ,q}}
\end{equation*}%
holds, where the implicit constant is independent of $R$. In addition%
\begin{equation*}
\big\|\eta _{R,m}\ast f\big\|_{\dot{K}_{p}^{n-\frac{n}{p},\infty }}\lesssim %
\big\|f\big\|_{\dot{K}_{p}^{n-\frac{n}{p},1}}
\end{equation*}%
and%
\begin{equation*}
\big\|\eta _{R,m}\ast f\big\|_{\dot{K}_{p}^{-\frac{n}{p},\infty }}\lesssim %
\big\|f\big\|_{\dot{K}_{p}^{-\frac{n}{p},1}.}
\end{equation*}
\end{lemma}

\begin{proof}
The method of the proof is very similar to that of the proof of Theorem 2.1
in \cite{LD96}.
\end{proof}

The following lemma plays an important role in this section.

\begin{lemma}
\label{Key-est1}\textit{Let }$\alpha ,\gamma \in \mathbb{R},1\leq p\leq
\infty ,0<s\leq \infty $ \textit{and }$R\geq H>0$\textit{. Then there exists
a constant }$c>0$ \textit{\ independent of }$R$ and $H$\textit{\ such that
for all }$f\in \dot{K}_{p}^{\alpha ,s}$, we have 
\begin{equation*}
\sup_{x\in B(0,\frac{1}{H})}\big(\eta _{R,N}\ast (|\cdot |^{-\gamma }f)(x)%
\big)\leq c\ \big(\frac{R}{H}\big)^{n}H^{\frac{n}{p}+\alpha +\gamma }\big\|f%
\big\|_{\dot{K}_{p}^{\alpha ,s}}
\end{equation*}%
for any $N$ large enough and $\alpha \leq n-\frac{n}{p}-\gamma $, where 
\begin{equation*}
s=\left\{ 
\begin{array}{ccc}
1, & \text{if} & \alpha =n-\frac{n}{p}-\gamma , \\ 
\infty , & \text{if} & \alpha <n-\frac{n}{p}-\gamma .%
\end{array}%
\right.
\end{equation*}
\end{lemma}

\begin{proof}
Let\ $x\in B(0,\frac{1}{H})$. Write 
\begin{equation*}
\eta _{R,N}\ast (|\cdot |^{-\gamma }f)\leq I_{1}+I_{2},
\end{equation*}%
where 
\begin{equation*}
I_{1}(x)=\int_{\overline{B(0,\frac{2}{H})}}|f(y)||y|^{-\gamma }\eta
_{R,N}(x-y)dy
\end{equation*}%
and%
\begin{equation*}
I_{2}(x)=\int_{\mathbb{R}^{n}\backslash \overline{B(0,\frac{2}{H})}%
}|f(y)||y|^{-\gamma }\eta _{R,N}(x-y)dy.
\end{equation*}%
Let $C(\frac{2^{l}}{H})=\{x\in \mathbb{R}^{n}:\frac{2^{l-2}}{H}<\left\vert
x\right\vert \leq \frac{2^{l-1}}{H}\},l\in \mathbb{Z}$. Using the following
decomposition 
\begin{align*}
I_{1}(x)=& \sum_{j=0}^{\infty }\int_{C(\frac{2^{2-j}}{H})}|f(y)||y|^{-\gamma
}\eta _{R,N}(x-y)dy, \\
I_{2}(x)=& \sum_{j=0}^{\infty }\int_{C(\frac{2^{j+3}}{H})}|f(y)||y|^{-\gamma
}\eta _{R,N}(x-y)dy
\end{align*}%
we obtain that $\eta _{R,N}\ast (|\cdot |^{-\gamma }f)$ can be estimated
from above by 
\begin{equation}
c\text{ }\sum_{j=0}^{\infty }\big(V_{j,R,H}^{1}+V_{j,R,H}^{2}\big),
\label{aux}
\end{equation}%
where 
\begin{equation*}
V_{j,R,H}^{1}=\eta _{R,dN}\ast |\cdot |^{-\gamma }|f|\chi _{C(\frac{2^{2-j}}{%
H})},\quad V_{j,R,H}^{2}=V_{-j-1,R,H}^{1}.
\end{equation*}

\textit{Step 1.} Let $\nu \in \mathbb{Z}$ be such that $2^{\nu -1}\leq
H^{-1}<2^{\nu }$. In this step\ we estimate the first sum of \eqref{aux}. A
simple change of variable and the H\"{o}lder inequality, yield for any $x\in
B(0,\frac{1}{H})$ 
\begin{align}
\sum_{j=0}^{\infty }V_{j,R,H}^{1}(x)\leq & R^{n}\sum_{j=0}^{\infty }\big\|%
|\cdot |^{-\gamma }f\chi _{\widetilde{C}(2^{2-j+\nu })}\big\|_{1}  \notag \\
\leq & R^{n}\sum\limits_{k-\infty }^{\nu }\big\||\cdot |^{-\gamma }f\chi _{%
\widetilde{C}_{k}}\big\|_{1}  \notag \\
\lesssim & R^{n}\sum\limits_{k-\infty }^{\nu }2^{k(n-\frac{n}{p}-\gamma )}%
\big\|f\chi _{\widetilde{C}_{k}}\big\|_{p}  \notag \\
\lesssim & \big(\frac{R}{H}\big)^{n}H^{\frac{n}{p}+\alpha +\gamma }\big\|f%
\big\|_{\dot{K}_{p}^{\alpha ,s}},  \label{aux3}
\end{align}%
where $\widetilde{C}_{k}=\{x\in \mathbb{R}^{n}:2^{k-2}\leq \left\vert
x\right\vert \leq 2^{k+2}\}$ and we have used the fact that $\alpha \leq n-%
\frac{n}{p}-\gamma $, $2^{k-v}\leq 1$ and 
\begin{equation*}
\sum\limits_{k-\infty }^{\nu }2^{(k-v)(n-\frac{n}{p}-\gamma -\alpha )}\leq c,
\end{equation*}%
whenever $\alpha <n-\frac{n}{p}-\gamma $, where the positive constant $c$ is
independent of $H$.

\textit{Step 2.} In this step\ we estimate the second sum of \eqref{aux}. We
see that for any $y\in C(\frac{2^{3+j}}{H})$ and any $x\in B(0,\frac{1}{H})$%
, we have $\left\vert x-y\right\vert >\frac{2^{j}}{H}$, so that 
\begin{equation*}
\eta _{R,N}(x-y)\leq R^{n}\big(\frac{2^{j}R}{H}\big)^{-N}\leq 2^{-jN}R^{n}.
\end{equation*}%
Hence by a simple change of variable and the H\"{o}lder inequality, we
obtain 
\begin{equation*}
\sum_{j=0}^{\infty }V_{j,R,H}^{2}(x)
\end{equation*}%
can be estimated from above by 
\begin{align*}
& R^{n}\sum_{j=0}^{\infty }2^{-jN}\big\||\cdot |^{-\gamma }f\chi _{C(\frac{%
2^{3+j}}{H})}\big\|_{1} \\
\lesssim & \text{ }R^{n}H^{-N}\sum\limits_{k=\nu }^{\infty }2^{-kN}\big\|%
|\cdot |^{-\gamma }f\chi _{\widetilde{C}_{k}}\big\|_{1} \\
\lesssim & \text{ }R^{n}H^{-N}\sum\limits_{k=\nu }^{\infty }2^{k(n-\frac{n}{p%
}-N-\gamma )}\big\|f\chi _{\widetilde{C}_{k}}\big\|_{p} \\
=& c\text{ }\big(\frac{R}{H}\big)^{n}H^{\frac{n}{p}+\alpha +\gamma
}\sum\limits_{k=\nu }^{\infty }2^{(k-\nu )(n-\frac{n}{p}-N-\alpha -\gamma
)}2^{k\alpha }\big\|f\chi _{\widetilde{C}_{k}}\big\|_{p}.
\end{align*}%
Since $N\ $large enough and $2^{k-\nu }\geq 1$, the right-hand side of the
last expression is bounded by 
\begin{equation}
c\text{ }\big(\frac{R}{H}\big)^{n}H^{\frac{n}{p}+\alpha +\gamma }\sup_{k\in 
\mathbb{Z}}\big(2^{k\alpha }\big\|f\chi _{\widetilde{C}_{k}}\big\|_{p}\big)%
\lesssim \text{ }\big(\frac{R}{H}\big)^{n}H^{\frac{n}{p}+\alpha +\gamma }%
\big\|f\big\|_{\dot{K}_{p}^{\alpha ,s}}.  \label{aux4.1}
\end{equation}%
Finally, the desired estimate follows from \eqref{aux3} and \eqref{aux4.1}
taking into account the decomposition \eqref{aux}.
\end{proof}

The following lemmas are key tools in this paper.

\begin{lemma}
\label{Bernstein-Herz-ine1}\textit{Let }$R>0,\alpha _{1},\alpha _{2},\gamma
\in \mathbb{R}\mathit{\ }$\textit{and} $1\leq q\leq p\leq \infty ,0<s\leq
\infty $. \textit{Suppose that }$\alpha _{1}+\frac{n}{p}>0$\textit{\ }and%
\begin{equation*}
\max \big(\alpha _{1}-\gamma ,-\frac{n}{q}\big)\leq \alpha _{2}<\min \big(n-%
\frac{n}{q}-\gamma ,n-\frac{n}{q}\big),\quad \alpha _{2}\neq -\frac{n}{q}.
\end{equation*}%
\textit{Then there exists a positive constant }$c>0$\textit{\ independent of 
}$R$\textit{\ such that for all }$f\in \dot{K}_{q}^{\alpha _{2},\theta }$%
\textit{, we have} 
\begin{equation*}
\big\|\eta _{R,N}\ast (|\cdot |^{-\gamma }f)\big\|_{\dot{K}_{p}^{\alpha
_{1},s}}\leq c\text{ }R^{\frac{n}{q}-\frac{n}{p}+\alpha _{2}-\alpha
_{1}+\gamma }\big\|f\big\|_{\dot{K}_{q}^{\alpha _{2},\theta }},
\end{equation*}%
where 
\begin{equation*}
\theta =\left\{ 
\begin{array}{ccc}
s, & \text{if} & \alpha _{2}=\alpha _{1}-\gamma , \\ 
\infty , & \text{if} & \alpha _{2}>\alpha _{1}-\gamma .%
\end{array}%
\right.
\end{equation*}
\end{lemma}

\begin{proof}
By similarity, we assume that $0<s<\infty $. Write 
\begin{equation}
\sum\limits_{k=-\infty }^{\infty }2^{k\alpha _{1}s}\big\|(\eta _{R,N}\ast
(|\cdot |^{-\gamma }f))\chi _{k}\big\|_{p}^{s}=I_{1,R}+I_{2,R},  \label{aux5}
\end{equation}%
with 
\begin{equation*}
I_{1,R}=\sum\limits_{k\in \mathbb{Z},2^{k}\leq \frac{1}{R}}2^{k\alpha _{1}s}%
\big\|(\eta _{R,N}\ast (|\cdot |^{-\gamma }f))\chi _{k}\big\|_{p}^{s}
\end{equation*}%
and%
\begin{equation*}
I_{2,R}=\sum\limits_{k\in \mathbb{Z},2^{k}>\frac{1}{R}}2^{k\alpha _{1}s}%
\big\|(\eta _{R,N}\ast (|\cdot |^{-\gamma }f))\chi _{k}\big\|_{p}^{s}
\end{equation*}%
We will estimate each term separately.

\textit{Estimate of }$I_{1,R}$. Lemma \ref{Key-est1}\ gives%
\begin{eqnarray*}
I_{1,R} &\leq &\sup_{x\in B(0,\frac{2}{R})}(\eta _{R,N}\ast (|\cdot
|^{-\gamma }f)(x))^{s}\sum\limits_{k\in \mathbb{Z},2^{k}\leq \frac{1}{R}%
}2^{k(\alpha _{1}+\frac{n}{p})s} \\
&\lesssim &R^{(\frac{n}{q}-\frac{n}{p}+\alpha _{2}-\alpha _{1}+\gamma )s}%
\big\|f\big\|_{\dot{K}_{q}^{\alpha _{2},\infty }}^{s},
\end{eqnarray*}%
because of $\alpha _{1}+\frac{n}{p}>0,\alpha _{2}<n-\frac{n}{q}-\gamma $ and 
$2^{k-1}R<1$.

\textit{Estimate of }$I_{2,R}$. Observe that%
\begin{eqnarray*}
\eta _{R,N}\ast \eta _{R,N}(z) &=&\int_{\mathbb{R}^{n}}\eta _{R,N}(z-h)\eta
_{R,N}(h)dh \\
&\geq &\eta _{R,N}(z)\int_{\mathbb{R}^{n}}\eta _{R,2N}(h)dh \\
&\geq &c\text{ }\eta _{R,N}(z)
\end{eqnarray*}%
for any $z\in \mathbb{R}^{n}$ and any $N>\frac{n}{2}$, where the positive
constant $c$ is independent of $z$ and $R$. Hence for any $x\in R_{k}$ 
\begin{align*}
\eta _{R,N}\ast (|\cdot |^{-\gamma }|f|)(x)\leq & c\text{ }\int_{\mathbb{R}%
^{n}}\eta _{R,N}\ast (|\cdot |^{-\gamma }|f|)(y)\eta _{R,N}(x-y)dy \\
\lesssim & V_{R,k}^{1}(x)+V_{R,k}^{2}(x)+V_{R,k}^{3}(x),
\end{align*}%
where the implicit constant is independent of $x,k$ and $R$, and 
\begin{equation*}
V_{R,k}^{1}(x)=\int_{B(0,2^{k-2})}\eta _{R,N}\ast (|\cdot |^{-\gamma
}|f|)(y)\eta _{R,N}(x-y)dy,
\end{equation*}%
\begin{equation*}
V_{R,k}^{2}(x)=\int_{\widetilde{C}_{k}}\eta _{R,N}\ast (|\cdot |^{-\gamma
}|f|)(y)\eta _{R,N}(x-y)dy
\end{equation*}%
and 
\begin{equation*}
V_{R,k}^{3}(x)=\int_{\mathbb{R}^{n}\backslash B(0,2^{k+2})}\eta _{R,N}\ast
(|\cdot |^{-\gamma }|f|)(y)\eta _{R,N}(x-y)dy
\end{equation*}%
It is easy to verify that if $x\in R_{k}$ and $y\in B(0,2^{k-2})$, then $%
\left\vert x-y\right\vert >2^{k-2}$. This estimate and Lemma \ref{Key-est1},
yield for any $x\in R_{k}$ and any $2^{k}R>1$ 
\begin{align*}
V_{R,k}^{1}(x)& \lesssim \text{ }\sup_{y\in B(0,2^{k})}\big(\eta _{R,N}\ast
(|\cdot |^{-\gamma }|f|)(y)\big)\int_{2^{k-2}<\left\vert z\right\vert
<2^{k+1}}\eta _{R,N}(z)dz \\
& \lesssim R^{2n-N}2^{-(\frac{n}{q}+\alpha _{2}+N+\gamma -2n)k}\big\|f\big\|%
_{\dot{K}_{q}^{\alpha _{2},\infty }},
\end{align*}%
where the positive constant $c$ is independent of $x,R,k$ and $f$. From
this,\ we get 
\begin{align*}
& \sum\limits_{k\in \mathbb{Z},2^{k}>\frac{1}{R}}2^{k\alpha _{1}s}\big\|%
V_{R,k}^{1}\chi _{k}\big\|_{p}^{s} \\
& \lesssim R^{(2n-N)s}\big\|f\big\|_{\dot{K}_{q}^{\alpha _{2},\infty
}}^{s}\sum\limits_{k\in \mathbb{Z},2^{k}>\frac{1}{R}}2^{k(\frac{n}{p}-\frac{n%
}{q}+2n+\alpha _{1}-\alpha _{2}-N-\gamma )s} \\
& \lesssim \text{ }R^{(\frac{n}{q}-\frac{n}{p}+\alpha _{2}-\alpha
_{1}+\gamma )s}\big\|f\big\|_{\dot{K}_{q}^{\alpha _{2},\infty }}^{s}.
\end{align*}%
We estimate $V_{R,k}^{2}$. Applying Young's inequality, we obtain 
\begin{equation*}
\big\|V_{R,k}^{2}\chi _{k}\big\|_{p}\leq c\text{ }R^{\frac{n}{q}-\frac{n}{p}}%
\big\|(\eta _{R,N}\ast (|\cdot |^{-\gamma }|f|))\chi _{\widetilde{C}_{k}}%
\big\|_{q},
\end{equation*}%
where we have used the fact that $N$ is large enough. We see that%
\begin{equation*}
\eta _{R,N}\ast (|\cdot |^{-\gamma }|f|)=J_{1,k}+J_{2,k}+J_{3,k},
\end{equation*}%
where 
\begin{equation*}
J_{1,k}(x)=\int_{B_{k-4}}|f(y)||y|^{-\gamma }\eta _{R,N}(x-y)dy,
\end{equation*}%
\begin{equation*}
J_{2,k}(x)=\int_{\widehat{C}_{k}}|f(y)||y|^{-\gamma }\eta _{R,N}(x-y)dy
\end{equation*}%
and%
\begin{equation*}
J_{3,k}(x)=\int_{\mathbb{R}^{n}\backslash B_{k+4}}|f(y)||y|^{-\gamma }\eta
_{R,N}(x-y)dy,
\end{equation*}%
with 
\begin{equation*}
\widehat{C}_{k}=\left\{ x\in \mathbb{R}^{n}:\,2^{k-4}\leq |x|\leq
2^{k+4}\right\} ,\quad k\in \mathbb{Z}.
\end{equation*}%
First, we have%
\begin{equation*}
J_{2,k}\lesssim 2^{-k\gamma }\eta _{R,N}\ast |f|.
\end{equation*}%
Hence%
\begin{equation*}
\Big(\sum\limits_{k\in \mathbb{Z},2^{k}>\frac{2}{R}}2^{k\alpha _{1}s}\big\|%
J_{2,k}\chi _{\widetilde{C}_{k}}\big\|_{q}^{s}\Big)^{\frac{1}{s}}
\end{equation*}%
can be estimated from above by%
\begin{eqnarray*}
&&c\Big(\sum\limits_{k\in \mathbb{Z},2^{k}>\frac{2}{R}}2^{k\left( \alpha
_{1}-\alpha _{2}-\gamma \right) s}2^{k\alpha _{2}s}\big\|(\eta _{R,N}\ast
|f|)\chi _{\widetilde{C}_{k}}\big\|_{q}^{s}\Big)^{\frac{1}{s}} \\
&\lesssim &\text{ }R^{\alpha _{2}-\alpha _{1}+\gamma }\sup_{k\in \mathbb{Z}}%
\big(2^{k\alpha _{2}}\big\|(\eta _{R,N}\ast f)\chi _{\widehat{C}_{k}}\big\|%
_{q}\big)\Big(\sum\limits_{k\in \mathbb{Z},2^{k}>\frac{2}{R}}\left(
2^{k}R\right) ^{\left( \alpha _{1}-\alpha _{2}-\gamma \right) s}\Big)^{\frac{%
1}{s}} \\
&\lesssim &\text{ }R^{\alpha _{2}-\alpha _{1}+\gamma }\big\|\eta _{R,N}\ast f%
\big\|_{\dot{K}_{q}^{\alpha _{2},\infty }} \\
&\lesssim &\text{ }R^{\alpha _{2}-\alpha _{1}+\gamma }\big\|f\big\|_{\dot{K}%
_{q}^{\alpha _{2},\infty }},
\end{eqnarray*}%
if $\alpha _{2}>\alpha _{1}-\gamma $, where we have used Lemma \ref%
{Maximal-Inq} since $-\frac{n}{q}<\alpha _{2}<n-\frac{n}{q}$. The case $%
\alpha _{2}=\alpha _{1}-\gamma $ can be easily solved. Now%
\begin{eqnarray*}
J_{1,k}(x) &=&\sum_{l=0}^{\infty }\int_{R_{k-4-l}}|f(y)||y|^{-\gamma }\eta
_{R,N}(x-y)dy \\
&\lesssim &R^{n-N}2^{-kN}\sum_{l=0}^{\infty }2^{(l-k)\gamma
}\int_{R_{k-4-l}}|f(y)|dy \\
&=&R^{n-N}2^{-kN}\sum_{v=-\infty }^{k}2^{-v\gamma }\int_{R_{v}}|f(y)|dy.
\end{eqnarray*}
Using H\"{o}lder's inequality, we get%
\begin{eqnarray*}
J_{1,k}(x) &\lesssim &R^{n-N}2^{-kN}\sum_{v=-\infty }^{k}2^{v(n-\frac{n}{q}%
-\gamma )}\big\|f\chi _{v}\big\|_{q} \\
&\lesssim &R^{n-N}2^{-kN}\big\|f\big\|_{\dot{K}_{q}^{\alpha _{2},\infty
}}\sum_{v=-\infty }^{k}2^{v(n-\frac{n}{q}-\alpha _{2}-\gamma )} \\
&\lesssim &R^{n-N}2^{k(n-\frac{n}{q}-\alpha _{2}-\gamma -N)}\big\|f\big\|_{%
\dot{K}_{q}^{\alpha _{2},\infty }}
\end{eqnarray*}%
because of $\alpha _{2}<n-\frac{n}{q}-\gamma $. Hence $\Big(%
\sum\limits_{k\in \mathbb{Z},2^{k}>\frac{2}{R}}2^{k\alpha _{1}r}\big\|%
J_{1,k}\chi _{\widehat{C}_{k}}\big\|_{q}^{s}\Big)^{\frac{1}{s}}$ is bounded
by%
\begin{equation*}
cR^{n-N}\big\|f\big\|_{\dot{K}_{q}^{\alpha _{2},\infty }}\Big(%
\sum\limits_{k\in \mathbb{Z},2^{k}>\frac{2}{R}}2^{k(n-\alpha _{2}-\gamma
-N+\alpha _{1})s}\Big)^{\frac{1}{s}}\lesssim R^{\alpha _{2}-\alpha
_{1}+\gamma }\big\|f\big\|_{\dot{K}_{q}^{\alpha _{2},\infty }}.
\end{equation*}%
Now%
\begin{eqnarray*}
J_{3,k}(x) &=&\sum_{l=0}^{\infty }\int_{R_{k+4+l}}|f(y)||y|^{-\gamma }\eta
_{R,N}(x-y)dy \\
&\lesssim &R^{n-N}\sum_{l=0}^{\infty }2^{-(k+l)(\gamma
+N)}\int_{R_{k+4+l}}|f(y)|dy \\
&\lesssim &R^{n-N}\sum_{v=k}^{\infty }2^{-v(\gamma +N-n+\frac{n}{q})}\big\|%
f\chi _{v}\big\|_{q} \\
&\lesssim &2^{-k(\gamma +N-n+\frac{n}{q}+\alpha _{2})}R^{n-N}\big\|f\big\|_{%
\dot{K}_{q}^{\alpha _{2},\infty }}.
\end{eqnarray*}%
This leads to $\Big(\sum\limits_{k\in \mathbb{Z},2^{k}>\frac{2}{R}%
}2^{k\alpha _{1}s}\big\|J_{3,k}\chi _{\widehat{C}_{k}}\big\|_{q}^{s}\Big)^{%
\frac{1}{s}}$ is bounded by%
\begin{equation*}
R^{n-N}\big\|f\big\|_{\dot{K}_{q}^{\alpha _{2},\infty }}\Big(%
\sum\limits_{k\in \mathbb{Z},2^{k}>\frac{2}{R}}2^{k(n-\alpha _{2}-\gamma
-N+\alpha _{1})r}\Big)^{\frac{1}{r}}\lesssim R^{\alpha _{2}-\alpha
_{1}+\gamma }\big\|f\big\|_{\dot{K}_{q}^{\alpha _{2},\infty }}.
\end{equation*}
For $V_{R,k}^{3}$, we see that, $V_{R,k}^{3}(x)$ can be estimated from above
by 
\begin{equation*}
c\sum\limits_{l=0}^{\infty }\int_{R_{k+l+3}}\eta _{R,N}\ast (|\cdot
|^{-\gamma }|f|)(y)\eta _{R,N}(x-y)dy
\end{equation*}%
for any $x\in R_{k}$. Since $\left\vert x-y\right\vert >3\cdot 2^{k+l}$ for
any $x\in R_{k}$ and any $y\in R_{k+l+3}$, the right-hand side of the last
term is bounded by 
\begin{align*}
& c\text{ }R^{n-N}\sum\limits_{l=0}^{\infty }2^{-(k+l)N}\big\|\eta
_{R,N}\ast (|\cdot |^{-\gamma }|f|)\chi _{C_{k+l+3}}\big\|_{1} \\
& =c\text{ }R^{n-N}\sum\limits_{j=k+3}^{\infty }2^{j(n-\frac{n}{q}-N)}\big\|%
\eta _{R,N}\ast (|\cdot |^{-\gamma }|f|)\chi _{C_{j}}\big\|_{q} \\
& \lesssim \text{ }R^{n-N}\sum\limits_{j=k+3}^{\infty }2^{j(n-N)}\sup_{x\in
B(0,2^{j})}(\eta _{R,N}\ast (|\cdot |^{-\gamma }|f|)(x)) \\
& \lesssim \text{ }R^{2n-N}\sum\limits_{j=k+3}^{\infty }2^{j(2n-N-\alpha
_{2}-\frac{n}{q}-\gamma )}\big\|f\big\|_{\dot{K}_{q}^{\alpha _{2},\infty }}
\\
& \lesssim \text{ }R^{2n-N}2^{k(2n-N-\alpha _{2}-\frac{n}{q}-\gamma )}\big\|f%
\big\|_{\dot{K}_{q}^{\alpha _{2},\infty }},
\end{align*}%
where we have used Lemma \ref{Key-est1} since $R>2^{-k}.$. Therefore, 
\begin{align*}
& \sum\limits_{k\in \mathbb{Z},2^{k}>\frac{1}{R}}2^{k\alpha _{1}s}\big\|%
V_{R,k}^{3}\chi _{k}\big\|_{p}^{s} \\
& \lesssim \text{ }R^{(2n-N)s}\big\|f\big\|_{\dot{K}_{q}^{\alpha
_{2},p}}^{s}\sum\limits_{k\in \mathbb{Z},2^{k}>\frac{1}{R}}2^{k(\frac{n}{p}%
+\alpha _{1}+2n-N-\alpha _{2}-\frac{n}{q}-\gamma )s} \\
& \lesssim \text{ }R^{(\frac{n}{q}-\frac{n}{p}+\alpha _{2}-\alpha
_{1}+\gamma )s}\big\|f\big\|_{\dot{K}_{q}^{\alpha _{2},\infty
}}^{s}\sum\limits_{k\in \mathbb{Z},2^{k}>\frac{1}{R}}\left( 2^{k}R\right) ^{(%
\frac{n}{p}+\alpha _{1}+2n-N-\alpha _{2}-\frac{n}{q}-\gamma )s} \\
& \lesssim \text{ }R^{(\frac{n}{q}-\frac{n}{p}+\alpha _{2}-\alpha
_{1}+\gamma )s}\big\|f\big\|_{\dot{K}_{q}^{\alpha _{2},\infty }}^{s},
\end{align*}%
where we have used the fact that $N$ is lage enough. The proof is complete.
\end{proof}

We now consider the case $\alpha _{2}=n-\frac{n}{q}-\gamma $ in Lemma \ref%
{Bernstein-Herz-ine1}.

\begin{lemma}
\label{Bernstein-Herz-ine1 copy(1)}\textit{Let }$R>0,0<\gamma <n,\alpha \in 
\mathbb{R}\mathit{\ }$\textit{and} $1\leq q\leq p\leq \infty ,0<s\leq \infty 
$. \textit{Suppose that }$\alpha +\frac{n}{p}>0$ and%
\begin{equation*}
\max (\alpha -\gamma ,-\frac{n}{q})<n-\frac{n}{q}-\gamma .
\end{equation*}%
\textit{Then there exists a positive constant }$c>0$\textit{\ independent of 
}$R$\textit{\ such that for all }$f\in \dot{K}_{q}^{n-\frac{n}{q}-\gamma ,1}$%
\textit{, we have} 
\begin{equation*}
\big\|\eta _{R,N}\ast (|\cdot |^{-\gamma }f)\big\|_{\dot{K}_{p}^{\alpha
,s}}\leq c\text{ }R^{n-\frac{n}{p}-\alpha }\big\|f\big\|_{\dot{K}_{q}^{n-%
\frac{n}{q}-\gamma ,1}}.
\end{equation*}%
In addition if $1\leq s\leq \infty $, then we have 
\begin{equation*}
\big\|\eta _{R,N}\ast (|\cdot |^{-\gamma }f)\big\|_{\dot{K}_{p}^{n-\frac{n}{q%
},s}}\leq c\text{ }R^{\frac{n}{q}-\frac{n}{p}}\big\|f\big\|_{\dot{K}_{q}^{n-%
\frac{n}{q}-\gamma ,1}}.
\end{equation*}
\end{lemma}

The case $p\leq q$ was not considered in the previous lemmas. The following
lemma establishes the corresponding result.

\begin{lemma}
\label{Bernstein-Herz-ine2}\textit{Let }$R>0,\gamma ,\alpha _{1},\alpha
_{2}\in \mathbb{R}\mathit{\ }$\textit{and} $0<\upsilon ,s\leq \infty $. 
\textit{Suppose that }$\alpha _{1}+\frac{n}{p}>0,1\leq p\leq q\leq \infty $.
Assume that%
\begin{equation*}
\max (\alpha _{1}-\gamma +\frac{n}{p}-\frac{n}{q},-\frac{n}{q})<\alpha
_{2}<\min \big(n-\frac{n}{q}-\gamma ,n-\frac{n}{q}\big)
\end{equation*}%
or%
\begin{equation*}
\alpha _{1}+\frac{n}{p}-\gamma >0,\alpha _{1}-\gamma +\frac{n}{p}-\frac{n}{q}%
=\alpha _{2}<\min \big(n-\frac{n}{q}-\gamma ,n-\frac{n}{q}\big)\quad \text{%
and}\quad s=\upsilon .
\end{equation*}%
\textit{Then there exists a positive constant }$c$\textit{\ independent of }$%
R$\textit{\ such that for all }$f\in \dot{K}_{q}^{\alpha _{2},\upsilon }$%
\textit{, we have} 
\begin{equation*}
\big\|\eta _{R,N}\ast (|\cdot |^{-\gamma }f)\big\|_{\dot{K}_{p}^{\alpha
_{1},s}}\leq c\text{ }R^{\frac{n}{q}-\frac{n}{p}+\alpha _{2}-\alpha
_{1}+\gamma }\big\|f\big\|_{\dot{K}_{q}^{\alpha _{2},\upsilon }}.
\end{equation*}
\end{lemma}

\begin{proof}
By similarity, we only consider the case $\alpha _{2}>\alpha _{1}-\gamma +%
\frac{n}{p}-\frac{n}{q}$ and assume that $0<s<\infty $. We employ the
notations $I_{1,R}$ and $I_{2,R}$ from \eqref{aux5}. The estimate of $%
I_{1,R} $ follows easily from Lemma \ref{Bernstein-Herz-ine1}. We only need
to estimate the part $I_{2,R}$. H\"{o}lder's inequality gives 
\begin{equation*}
\big\|(\eta _{R,N}\ast (|\cdot |^{-\gamma }f))\chi _{k}\big\|_{p}\leq c\text{
}2^{k(\frac{n}{p}-\frac{n}{q})}\big\|(\eta _{R,N}\ast (|\cdot |^{-\gamma
}f))\chi _{k}\big\|_{q}.
\end{equation*}%
where the positive constant $c$ is independent of $k$. Therefore, 
\begin{equation*}
I_{2,R}\lesssim \sum\limits_{k\in \mathbb{Z},2^{k}>\frac{1}{R}}2^{k(\frac{n}{%
p}-\frac{n}{q}-\alpha _{2}+\alpha _{1})s}2^{k\alpha _{2}s}\big\|(\eta
_{R,N}\ast (|\cdot |^{-\gamma }f))\chi _{k}\big\|_{q}^{r}.
\end{equation*}%
In view of the estimation of $V_{R,k}^{2}$ in Lemma \ref{Bernstein-Herz-ine1}%
we have%
\begin{eqnarray*}
\eta _{R,N}\ast (|\cdot |^{-\gamma }|f|)\chi _{k} &=&J_{1,k}\chi
_{k}+J_{2,k}\chi _{k}+J_{3,k}\chi _{k} \\
&\lesssim &2^{-k\gamma }\eta _{R,N}\ast |f|\chi _{k}+2^{-k(\gamma +N-n+\frac{%
n}{q}+\alpha _{2})}R^{n-N}\big\|f\big\|_{\dot{K}_{q}^{\alpha _{2},\infty }}.
\end{eqnarray*}
Observing that 
\begin{eqnarray*}
&&\sum\limits_{k\in \mathbb{Z},2^{k}>\frac{1}{R}}2^{k(\frac{n}{p}-\frac{n}{q}%
-\alpha _{2}+\alpha _{1}-\gamma )s}2^{k\alpha _{2}s}\big\|(\eta _{R,N}\ast
f)\chi _{k}\big\|_{q}^{s} \\
&\lesssim &R^{(\frac{n}{q}-\frac{n}{p}+\alpha _{2}-\alpha _{1}+\gamma )s}%
\big\|\eta _{R,N}\ast f\big\|_{\dot{K}_{q}^{\alpha _{2},\infty }}^{s} \\
&\lesssim &R^{(\frac{n}{q}-\frac{n}{p}+\alpha _{2}-\alpha _{1}+\gamma )s}%
\big\|f\big\|_{\dot{K}_{q}^{\alpha _{2},\infty }}^{s}
\end{eqnarray*}%
and%
\begin{eqnarray*}
&&R^{(n-N)s}\sum\limits_{k\in \mathbb{Z},2^{k}>\frac{1}{R}}2^{k(\frac{n}{p}-%
\frac{n}{q}-\alpha _{2}+\alpha _{1}-\gamma -N+n)s}\lesssim R^{(\frac{n}{q}-%
\frac{n}{p}+\alpha _{2}-\alpha _{1}+\gamma )s} \\
&\lesssim &R^{(\frac{n}{q}-\frac{n}{p}+\alpha _{2}-\alpha _{1}+\gamma )s}.
\end{eqnarray*}
Then%
\begin{equation*}
I_{2,R}\lesssim R^{(\frac{n}{q}-\frac{n}{p}+\alpha _{2}-\alpha _{1}+\gamma
)s}\big\|f\big\|_{\dot{K}_{q}^{\alpha _{2},\upsilon }}^{s}.
\end{equation*}%
The proof is complete.
\end{proof}

\begin{lemma}
\label{Bernstein-Herz-ine2 copy(1)}\textit{Let }$R>0,0<\gamma <n,\alpha \in 
\mathbb{R}\mathit{\ }$\textit{and} $1\leq s\leq \infty $. Assume that $1\leq
p\leq q\leq \infty $ and%
\begin{equation*}
-\frac{n}{p}<\alpha <n-\frac{n}{p}.
\end{equation*}%
\textit{Then there exists a positive constant }$c$\textit{\ independent of }$%
R$\textit{\ such that for all }$f\in \dot{K}_{q}^{n-\frac{n}{q}-\gamma ,1}$%
\textit{, we have} 
\begin{equation*}
\big\|\eta _{R,N}\ast (|\cdot |^{-\gamma }f)\big\|_{\dot{K}_{p}^{\alpha
,s}}\leq c\text{ }R^{n-\frac{n}{p}-\alpha }\big\|f\big\|_{\dot{K}_{q}^{n-%
\frac{n}{q}-\gamma ,1}}.
\end{equation*}%
in addition%
\begin{equation*}
\big\|\eta _{R,N}\ast (|\cdot |^{-\gamma }f)\big\|_{\dot{K}_{p}^{n-\frac{n}{p%
},1}}\leq c\text{ }\big\|f\big\|_{\dot{K}_{q}^{n-\frac{n}{q}-\gamma ,1}}.
\end{equation*}
\end{lemma}

\begin{proof}
The proof follows the method of Lemma \ref{Bernstein-Herz-ine2}, together
with Lemma \ref{Bernstein-Herz-ine1 copy(1)}.
\end{proof}

\section{Local well-posedness}

We study equation \eqref{HH-equation}\ through its corresponding integral
formulation. A function $u$ is called a mild solution of \eqref{HH-equation}
if it satisfies 
\begin{equation}
u(t)=e^{t\Delta }u_{0}+a\int_{0}^{t}e^{(t-\tau )\Delta }\bigl(|\cdot
|^{-\gamma }|u(\tau )|^{\beta }u(\tau )\bigr)d\tau .  \label{int-equa1}
\end{equation}

Throughout this paper, for $\beta >0$, $\alpha ,a\in \mathbb{R}$, and $%
\gamma <2-\alpha \beta $, we define the critical exponent by 
\begin{equation}
q_{c}=\frac{n\beta }{2-\alpha \beta -\gamma }.  \label{criticalnumber}
\end{equation}%
The exponent $q_{c}$\ plays a fundamental role in our analysis. Accordingly,
we classify the exponent $q$ as subcritical, critical, or supercritical
according to whether $1\leq q<q_{c},q=q_{c},$ \text{or} $q>q_{c},$
respectively. Notice that \eqref{HH-equation} in Lebesgue spaces was studied
in \cite{BTW17} and in Lebesgue spaces with power weighted; see \cite{CIT22}%
. In weighted Lorentz spaces we refer to \cite{CITT26}. All the results in
this section generalize the corresponding results obtained in \cite{BTW17}.

We are now in a position to state the first main result of this paper.

\begin{theorem}
	\label{result1}
	Let $n\geq1$ be an integer, $\alpha\in\mathbb{R}$, $1\leq r\leq\infty$,
	$\beta>0$, and let $\gamma$ satisfy
	\begin{equation*}
	-\alpha\beta\leq\gamma<\min(n,2-\alpha\beta).
	\end{equation*}
	Let $1\leq q<\infty$ be such that
	\begin{equation*}
	-\frac{\alpha}{n}<\frac{1}{q}
	<
	\min\left(
	\frac{n-\gamma}{n(\beta+1)}-\frac{\alpha}{n},
	\frac{1}{\beta+1}-\frac{\alpha}{n},
	\frac{1}{\beta+1}
	\right).
	\end{equation*}
	Let $q_c$ be defined by \eqref{criticalnumber}, and assume that $q>q_c$.
	Then the integral equation \eqref{int-equa1} is locally well-posed in
	$\dot{K}_{q}^{\alpha,r}$. More precisely, for every
	$u_0\in\dot{K}_{q}^{\alpha,r}$, there exist a time $T>0$ and a unique
	solution
	\[
	u\in C([0,T],\dot{K}_{q}^{\alpha,r})
	\]
	of \eqref{int-equa1}. Furthermore, this solution can be extended to a
	maximal interval of existence $[0,T_{\max})$, where either
	$T_{\max}=\infty$, or $T_{\max}<\infty$ and
	\[
	\lim_{t\to T_{\max}}
	\|u(t)\|_{\dot{K}_{q}^{\alpha,r}}
	=\infty.
	\]
\end{theorem}

\begin{proof}
	First, our assumptions on $\gamma$ and $q$ imply that
	\begin{equation*}
	\max \Big(\frac{\alpha-\gamma}{\beta+1},-\frac{n}{q}\Big)
	\leq \alpha
	<
	\min \Big(
	\frac{n-\gamma}{\beta+1}-\frac{n}{q},
	\frac{n}{\beta+1}-\frac{n}{q}
	\Big),
	\end{equation*}
	and, in particular,
	\[
	\alpha>-\frac{n}{q}.
	\]
	Indeed, these inequalities are consistent since
	\[
	-\alpha\beta\leq\gamma<\min(n,2-\alpha\beta),
	\]
	and
	\[
	\frac{n-\gamma}{\beta+1}-\alpha
	\leq
	\frac{n+\alpha\beta}{\beta+1}-\alpha
	=
	\frac{n-\alpha}{\beta+1}.
	\]
	Hence, Lemma~\ref{Bernstein-Herz-ine1} is applicable.
	
	For $t>0$, define
	\[
	\digamma_t(u)
	=
	e^{t\Delta}\big(|\cdot|^{-\gamma}|u|^{\beta}u\big).
	\]
	Using the embedding
	\[
	\dot{K}_{q}^{\alpha,\frac{r}{\beta+1}}
	\hookrightarrow
	\dot{K}_{q}^{\alpha,r}
	\]
	and Lemma~\ref{Bernstein-Herz-ine1}, we obtain
	\begin{align*}
	\|\digamma_t(u)-\digamma_t(v)\|_{\dot{K}_{q}^{\alpha,r}}
	&\lesssim
	\|\digamma_t(u)-\digamma_t(v)\|_{\dot{K}_{q}^{\alpha,\frac{r}{\beta+1}}}
	\\
	&\lesssim
	t^{-\frac12\left(\frac{n\beta}{q}+\alpha\beta+\gamma\right)}
	\bigl\|
	|u|^{\beta}u-|v|^{\beta}v
	\bigr\|_{\dot{K}_{\frac{q}{\beta+1}}^{\alpha(\beta+1),\frac{r}{\beta+1}}}.
	\end{align*}
	Using the elementary inequality
	\[
	\bigl||u|^{\beta}u-|v|^{\beta}v\bigr|
	\leq
	C\bigl(|u|^{\beta}+|v|^{\beta}\bigr)|u-v|,
	\]
	we deduce that
	\begin{align*}
	&
	\|\digamma_t(u)-\digamma_t(v)\|_{\dot{K}_{q}^{\alpha,r}}
	\\
	&\lesssim
	t^{-\frac12\left(\frac{n\beta}{q}+\alpha\beta+\gamma\right)}
	\bigl\|
	(|u|^{\beta}+|v|^{\beta})(u-v)
	\bigr\|_{\dot{K}_{\frac{q}{\beta+1}}^{\alpha(\beta+1),\frac{r}{\beta+1}}}
	\\
	&\lesssim
	t^{-\frac12\left(\frac{n\beta}{q}+\alpha\beta+\gamma\right)}
	\bigl\|
	|u|^{\beta}+|v|^{\beta}
	\bigr\|_{\dot{K}_{\frac{q}{\beta}}^{\alpha\beta,\frac{r}{\beta}}}
	\|u-v\|_{\dot{K}_{q}^{\alpha,r}}
	\\
	&\lesssim
	t^{-\frac12\left(\frac{n\beta}{q}+\alpha\beta+\gamma\right)}
	\left(
	\|u\|_{\dot{K}_{q}^{\alpha,r}}^{\beta}
	+
	\|v\|_{\dot{K}_{q}^{\alpha,r}}^{\beta}
	\right)
	\|u-v\|_{\dot{K}_{q}^{\alpha,r}}
	\\
	&\lesssim
	M^{\beta}
	t^{-\frac12\left(\frac{n\beta}{q}+\alpha\beta+\gamma\right)}
	\|u-v\|_{\dot{K}_{q}^{\alpha,r}},
	\end{align*}
	whenever
	\[
	\|u\|_{\dot{K}_{q}^{\alpha,r}},
	\,
	\|v\|_{\dot{K}_{q}^{\alpha,r}}
	\leq M.
	\]
	Here, the second inequality follows from H\"{o}lder's inequality, since
	\[
	\frac{\beta+1}{q}
	=
	\frac{\beta}{q}
	+
	\frac{1}{q}.
	\]
	
	Since
	\[
	q>q_c=\frac{n\beta}{2-\alpha\beta-\gamma},
	\]
	it follows that
	\[
	t^{-\frac12\left(\frac{n\beta}{q}+\alpha\beta+\gamma\right)}
	\in L_{\mathrm{loc}}^{1}(0,\infty).
	\]
	Moreover,
	\[
	\|\digamma_t(0)\|_{\dot{K}_{q}^{\alpha,r}}=0,
	\]
	and the semigroup property gives
	\[
	e^{s\Delta}\digamma_t=\digamma_{t+s},
	\qquad s,t>0.
	\]
	Therefore, all the hypotheses of \cite[Theorem~1, p.~279]{We80} are satisfied, and the conclusion follows.
\end{proof}

\begin{remark}
	When $\alpha=0$, Theorem~\ref{result1} reduces to the corresponding result established in \cite{BTW17}.
\end{remark}

We now consider the case \(q=\infty\) in Theorem~\ref{result1}.

\begin{theorem}
	\label{result1 copy(1)}
	Let $n\ge1$ be an integer, $1\le r\le\infty$, $\gamma<n$, and $\beta>0$. Let
	$\alpha>0$ satisfy
	\[
	\frac{-\gamma}{\beta}
	\le
	\alpha
	<
	\min\Big(
	\frac{2-\gamma}{\beta},
	\frac{n-\gamma}{\beta+1},
	\frac{n}{\beta+1}
	\Big).
	\]
	Then \eqref{int-equa1} is locally well posed in
	$\dot{K}_{\infty}^{\alpha,r}$. More precisely, for every
	$u_{0}\in\dot{K}_{\infty}^{\alpha,r}$, there exists $T>0$ such that
	\eqref{int-equa1} admits a unique solution
	\[
	u\in C([0,T],\dot{K}_{\infty}^{\alpha,r}).
	\]
	Moreover, the solution can be extended to a maximal interval of existence
	$[0,T_{\max})$, where either $T_{\max}=\infty$, or
	$T_{\max}<\infty$ and
	\[
	\lim_{t\to T_{\max}}
	\|u(t)\|_{\dot{K}_{\infty}^{\alpha,r}}
	=\infty.
	\]
\end{theorem}
\begin{proof}
	By the assumption on $\alpha$, we have
	\[
	-\alpha\beta \leq \gamma < \min(n,\,2-\alpha\beta)
	\quad\text{and}\quad
	\alpha \geq \frac{\alpha-\gamma}{\beta+1}.
	\]
	Using the continuous embedding
	\[
	\dot{K}_{\infty}^{\alpha,\frac{r}{\beta+1}}
	\hookrightarrow
	\dot{K}_{\infty}^{\alpha,r}
	\]
	together with Lemma~\ref{Bernstein-Herz-ine1}, we obtain
	\begin{align*}
	\|\digamma_t(u)-\digamma_t(v)\|_{\dot{K}_{\infty}^{\alpha,r}}
	&\lesssim
	\|\digamma_t(u)-\digamma_t(v)\|_{\dot{K}_{\infty}^{\alpha,\frac{r}{\beta+1}}}
	\\
	&\lesssim
	t^{-\frac12(\alpha\beta+\gamma)}
	\bigl\||u|^\beta u-|v|^\beta v\bigr\|_
	{\dot{K}_{\infty}^{\alpha(\beta+1),\frac{r}{\beta+1}}}
	\\
	&\lesssim
	t^{-\frac12(\alpha\beta+\gamma)}
	\left(
	\|u\|_{\dot{K}_{\infty}^{\alpha,r}}^\beta
	+
	\|v\|_{\dot{K}_{\infty}^{\alpha,r}}^\beta
	\right)
	\|u-v\|_{\dot{K}_{\infty}^{\alpha,r}},
	\end{align*}
	whenever
	\[
	\|u\|_{\dot{K}_{\infty}^{\alpha,r}}\le M
	\quad\text{and}\quad
	\|v\|_{\dot{K}_{\infty}^{\alpha,r}}\le M.
	\]
	
	Since $\gamma<2-\alpha\beta$, it follows that
	\[
	t^{-\frac12(\alpha\beta+\gamma)}
	\in L_{\mathrm{loc}}^1(0,\infty).
	\]
	Moreover,
	\[
	\|\digamma_t(0)\|_{\dot{K}_{\infty}^{\alpha,r}}=0,
	\]
	and the semigroup property
	\[
	e^{t\Delta}\digamma_s=\digamma_{t+s},
	\qquad s,t>0,
	\]
	holds. Therefore, the conclusion follows from
	\cite[Theorem~1, p.~279]{We80}.
\end{proof}

We now study the limiting case
\[
q=\frac{1}{\frac{n-\gamma}{n(\beta+1)}-\frac{\alpha}{n}}
\]
of Theorem~\ref{result1}.

\begin{theorem}
	\label{result2}
	Let $n\geq 1$ be an integer, $1\leq q<\infty$, $\beta>0$, and $\gamma$ satisfy
	\[
	0<\gamma<\min\bigl(n,\,2-(n-2)\beta\bigr).
	\]
	Assume that $q\geq\beta+1$. Then \eqref{int-equa1} is locally well-posed in
	\[
	\dot{K}_{q}^{\frac{n-\gamma}{\beta+1}-\frac{n}{q},\,1}.
	\]
	More precisely, for every
	\[
	u_{0}\in
	\dot{K}_{q}^{\frac{n-\gamma}{\beta+1}-\frac{n}{q},\,1},
	\]
	there exist $T>0$ and a unique solution
	\[
	u\in C ([0,T],\dot{K}_{q}^{\frac{n-\gamma}{\beta+1}-\frac{n}{q},\,1})
	\]
	of \eqref{int-equa1}. Moreover, the solution can be extended to a maximal interval
	$[0,T_{\max})$, where either
	\[
	T_{\max}=\infty,
	\]
	or
	\[
	T_{\max}<\infty
	\quad\text{and}\quad
	\lim_{t\to T_{\max}}
	\|u(t)\|_{\dot{K}_{q}^{\frac{n-\gamma}{\beta+1}-\frac{n}{q},\,1}}
	=\infty.
	\]
\end{theorem}
\begin{proof}
	First, the assumptions $q\geq \beta+1$, $\beta>0$, and
	\[
	0<\gamma<\min\bigl(n,2-(n-2)\beta\bigr)
	\]
	imply that
	\[
	\frac{n-\gamma}{\beta+1}-\frac{n}{q}-\gamma
	<
	n-\gamma-\frac{n(\beta+1)}{q},
	\]
	and
	\[
	\frac{n-\gamma}{\beta+1}-\frac{n}{q}
	>
	-\frac{n}{q}.
	\]
	Hence, all the assumptions of Lemma~\ref{Bernstein-Herz-ine1 copy(1)} are satisfied. Therefore,
	\begin{align}
	\|\digamma_t(u)-\digamma_t(v)\|_{\dot{K}_q^{\frac{n-\gamma}{\beta+1}-\frac{n}{q},1}}
	&\lesssim
	t^{-\frac{n\beta+\gamma}{2(\beta+1)}}
	\||u|^\beta u-|v|^\beta v\|_{
		\dot{K}_{\frac{q}{\beta+1}}^{
			\left(\frac{n-\gamma}{\beta+1}-\frac{n}{q}\right)(\beta+1),
			\frac{1}{\beta+1}}}
	\nonumber\\
	&\lesssim
	t^{-\frac{n\beta+\gamma}{2(\beta+1)}}
	\|(|u|^\beta+|v|^\beta)(u-v)\|_{
		\dot{K}_{\frac{q}{\beta+1}}^{
			\left(\frac{n-\gamma}{\beta+1}-\frac{n}{q}\right)(\beta+1),
			\frac{1}{\beta+1}}},
	\label{main5}
	\end{align}
	where we also used the embedding
	\[
	\dot{K}_q^{\frac{n-\gamma}{\beta+1}-\frac{n}{q},\frac1{\beta+1}}
	\hookrightarrow
	\dot{K}_q^{\frac{n-\gamma}{\beta+1}-\frac{n}{q},1}.
	\]
	
	Applying H\"older's inequality yields
	\begin{align*}
	\eqref{main5}
	&\lesssim
	t^{-\frac{n\beta+\gamma}{2(\beta+1)}}
	\||u|^\beta+|v|^\beta\|_{
		\dot{K}_{\frac{q}{\beta}}^{
			\left(\frac{n-\gamma}{\beta+1}-\frac{n}{q}\right)\beta,
			\frac1\beta}}
	\|u-v\|_{\dot{K}_q^{\frac{n-\gamma}{\beta+1}-\frac{n}{q},1}}
	\\
	&\lesssim
	t^{-\frac{n\beta+\gamma}{2(\beta+1)}}
	\left(
	\|u\|_{\dot{K}_q^{\frac{n-\gamma}{\beta+1}-\frac{n}{q},1}}^\beta
	+
	\|v\|_{\dot{K}_q^{\frac{n-\gamma}{\beta+1}-\frac{n}{q},1}}^\beta
	\right)
	\|u-v\|_{\dot{K}_q^{\frac{n-\gamma}{\beta+1}-\frac{n}{q},1}}.
	\end{align*}
	
	Assuming that
	\[
	\|u\|_{\dot{K}_q^{\frac{n-\gamma}{\beta+1}-\frac{n}{q},1}}
	\le M,
	\qquad
	\|v\|_{\dot{K}_q^{\frac{n-\gamma}{\beta+1}-\frac{n}{q},1}}
	\le M,
	\]
	we conclude that
	\[
	\|\digamma_t(u)-\digamma_t(v)\|_{\dot{K}_q^{\frac{n-\gamma}{\beta+1}-\frac{n}{q},1}}
	\lesssim
	M^\beta
	t^{-\frac{n\beta+\gamma}{2(\beta+1)}}
	\|u-v\|_{\dot{K}_q^{\frac{n-\gamma}{\beta+1}-\frac{n}{q},1}}.
	\]
	
	Moreover, since
	\[
	0<\gamma<2-(n-2)\beta,
	\]
	we have
	\[
	t^{-\frac{n\beta+\gamma}{2(\beta+1)}}
	\in L_{\mathrm{loc}}^{1}(0,\infty).
	\]
	Finally,
	\[
	\|\digamma_t(0)\|_{\dot{K}_q^{\frac{n-\gamma}{\beta+1}-\frac{n}{q},1}}=0,
	\]
	and the semigroup property
	\[
	e^{t\Delta}\digamma_s=\digamma_{t+s},
	\qquad t,s>0,
	\]
	holds. Therefore, the conclusion follows from
	\cite[Theorem~1, p.~279]{We80}.
\end{proof}

In the following, we study \eqref{int-equa1} in the homogeneous Herz space $\dot{K}_{q}^{\alpha,r}$ for the supercritical case $q>q_{c}$.

\begin{theorem}
	\label{result3}
	Let $n\geq 1$ be an integer, $1\leq q<\infty$, $1\leq r<\infty$,
	$\beta>0$, and $\gamma\in\mathbb{R}$ satisfy
	\[
	\gamma<\min(2,n).
	\]
	Let $\kappa\geq1$ be such that
	\begin{equation}
	\max\Big(\frac{1}{q}-\frac{2}{n(\beta+1)},\,
	\frac{1}{q(\beta+1)},\,
	\frac{1}{q}-\frac{2-\gamma}{n\beta}\Big)
	\leq
	\frac{1}{\kappa}
	<
	\min\Big(\frac{1}{q},\,
	\frac{n-\gamma}{n(\beta+1)}+\frac{\gamma}{n\beta},\,
	\frac{1}{\beta+1}+\frac{\gamma}{n\beta}\Big),
	\label{assu-result31}
	\end{equation}
	with
	\[
	\frac{1}{\kappa}>\frac{1}{q}-\frac{2}{n(\beta+1)}.
	\]
	Assume further that $\alpha\in\mathbb{R}$ satisfies
	$\alpha>-\frac{n}{\kappa}$ and
	\begin{equation}
	\max\Big(-\frac{\gamma}{\beta},-\frac{n}{\kappa}\Big)
	\leq
	\alpha
	<
	\min\Big(
	\frac{n-\gamma}{\beta+1}-\frac{n}{\kappa},\,
	\frac{n}{\beta+1}-\frac{n}{\kappa},\,
	\frac{2-\gamma}{\beta}-\frac{n}{q}
	\Big).
	\label{assu-result3}
	\end{equation}
	Then \eqref{int-equa1} is locally well posed in
	$\dot{K}_{q}^{\alpha,r}$, as in Theorem~\ref{result1}, except that
	uniqueness holds only in the class of continuous functions $u$ for which
	\[
	t^{\frac{n}{2q}-\frac{n}{2\kappa}}
	\|u(t)\|_{\dot{K}_{\kappa}^{\alpha,r}}
	\]
	remains bounded on $(0,T]$.
	
	Moreover, the solution can be extended to a maximal interval
	$[0,T_{\max})$ such that either $T_{\max}=\infty$, or
	$T_{\max}<\infty$ and
	\[
	\lim_{t\to T_{\max}}
	\|u(t)\|_{\dot{K}_{q}^{\alpha,r}}
	=\infty.
	\]
	Furthermore, the minimal existence time $T$ depends only on
	$\|u_{0}\|_{\dot{K}_{q}^{\alpha,r}}$.
\end{theorem}

\begin{proof}
First observe that, since
\[
\alpha <\frac{2-\gamma}{\beta}-\frac{n}{q},
\]
we have $q>q_{c}$. Set
\[
\delta =\frac{n}{2q}-\frac{n}{2\kappa}.
\]
Let $M,\sigma,T>0$ and let $u_{0}\in \dot{K}_{q}^{\alpha,r}$ satisfy
\begin{equation}
\big\|u_{0}\big\|_{\dot{K}_{q}^{\alpha,r}}\leq \sigma,\quad
c_{2}M^{\beta}T^{1-\frac{1}{2}\left(\frac{n\beta}{q}+\alpha\beta+\gamma\right)}<1,
\label{assu1}
\end{equation}
and
\begin{equation}
\sigma+c_{2}M^{\beta+1}
T^{1-\frac{1}{2}\left(\frac{n\beta}{q}+\alpha\beta+\gamma\right)}
\leq M,
\label{assu2}
\end{equation}
where $c_{2}>0$ is a positive constant. The proof is based on the Banach contraction principle.

Define
\[
F_{M}=\Big\{
u\in C([0,T],\dot{K}_{q}^{\alpha,r})
\cap C([0,T],\dot{K}_{\kappa}^{\alpha,r})
:\,
\max\Big(
\sup_{t\in[0,T]}\|u(t)\|_{\dot{K}_{q}^{\alpha,r}},
\sup_{t\in[0,T]}t^{\delta}\|u(t)\|_{\dot{K}_{\kappa}^{\alpha,r}}
\Big)\leq M
\Big\}.
\]

Endow $F_{M}$ with the metric
\[
d(u,v)=
\max\Big(
\sup_{t\in[0,T]}
\|u(t)-v(t)\|_{\dot{K}_{q}^{\alpha,r}},
\sup_{t\in[0,T]}
t^{\delta}
\|u(t)-v(t)\|_{\dot{K}_{\kappa}^{\alpha,r}}
\Big).
\]

Then $(F_{M},d)$ is a complete metric space. We shall prove that the integral equation \eqref{int-equa1} admits a unique solution in $F_{M}$.

Let $u,v\in F_{M}$ and $u_{0}, v_{0}\in \dot{K}_{q}^{\alpha,r}$. Since $\kappa>q$ and
\[
\alpha<\frac{n}{\beta+1}-\frac{n}{q(\beta+1)}
\]
(see \eqref{assu-result3}), it follows that
\[
-\frac{n}{q}<\alpha<n-\frac{n}{q}.
\] 
By Lemma \ref{Bernstein-Herz-ine1}, we have
\begin{equation*}
\big\|e^{t\Delta}(u_{0}-v_{0})\big\|_{\dot{K}_{q}^{\alpha,r}}
\lesssim
\big\|u_{0}-v_{0}\big\|_{\dot{K}_{q}^{\alpha,r}}.
\end{equation*}

Define
\begin{equation*}
\Lambda(u)(t)
=
a\int_{0}^{t}
e^{(t-\tau)\Delta}
\big(|\cdot|^{-\gamma}|u(\tau)|^{\beta}u(\tau)\big)
\,d\tau.
\end{equation*}
Then
\[
\big\|\Lambda(u)(t)-\Lambda(v)(t)\big\|_{\dot{K}_{q}^{\alpha,r}}
\]
is bounded by
\begin{equation}
|a|
\int_{0}^{t}
\big\|
e^{(t-\tau)\Delta}
\big(
|\cdot|^{-\gamma}
\big(|u(\tau)|^{\beta}u(\tau)-|v(\tau)|^{\beta}v(\tau)\big)
\big)
\big\|_{\dot{K}_{q}^{\alpha,r}}
\,d\tau.
\label{int1}
\end{equation}

Using the continuous embedding
\[
\dot{K}_{q}^{\alpha,\frac{r}{\beta+1}}
\hookrightarrow
\dot{K}_{q}^{\alpha,r},
\]
Lemma \ref{Bernstein-Herz-ine1} with
\[
(\alpha_1,\alpha_2,p,q,s,\theta)
=
\left(
\alpha,
\alpha(\beta+1),
q,
\frac{\kappa}{\beta+1},
\frac{r}{\beta+1},
\frac{r}{\beta+1}
\right),
\]
(see \eqref{assu-result3}), and H\"{o}lder's inequality, we obtain
\begin{align*}
&
\big\|
e^{(t-\tau)\Delta}
\big(
|\cdot|^{-\gamma}
\big(|u(\tau)|^{\beta}u(\tau)-|v(\tau)|^{\beta}v(\tau)\big)
\big)
\big\|_{\dot{K}_{q}^{\alpha,r}}
\\
&\lesssim
(t-\tau)^{-\frac12
	\left(
	\frac{n(\beta+1)}{\kappa}
	-\frac{n}{q}
	+\alpha\beta+\gamma
	\right)}
\big\|
|u(\tau)|^{\beta}u(\tau)
-
|v(\tau)|^{\beta}v(\tau)
\big\|_{
	\dot{K}_{\frac{\kappa}{\beta+1}}^{
		\alpha(\beta+1),
		\frac{r}{\beta+1}
}}
\\
&\lesssim
(t-\tau)^{-\frac12
	\left(
	\frac{n(\beta+1)}{\kappa}
	-\frac{n}{q}
	+\alpha\beta+\gamma
	\right)}
\big\|
|u(\tau)|^{\beta}
+
|v(\tau)|^{\beta}
\big\|_{
	\dot{K}_{\frac{\kappa}{\beta}}^{
		\alpha\beta,
		\frac{r}{\beta}
}}
\,
\big\|
u(\tau)-v(\tau)
\big\|_{\dot{K}_{\kappa}^{\alpha,r}}
\\
&\le
2cM^{\beta}
(t-\tau)^{-\frac12
	\left(
	\frac{n(\beta+1)}{\kappa}
	-\frac{n}{q}
	+\alpha\beta+\gamma
	\right)}
\tau^{-\delta(\beta+1)}
\,d(u,v).
\end{align*}
Consequently, \eqref{int1} is bounded by
\begin{eqnarray*}
	&&
	2|a|cM^{\beta}
	\int_{0}^{t}
	(t-\tau)^{-\frac12
		\left(
		\frac{n(\beta+1)}{\kappa}
		-\frac{n}{q}
		+\alpha\beta+\gamma
		\right)}
	\tau^{-\delta(\beta+1)}
	\,d\tau\,
	d(u,v)
	\\
	&=&
	2|a|cM^{\beta}\vartheta\,
	t^{-\frac12
		\left(
		\frac{n(\beta+1)}{\kappa}
		-\frac{n}{q}
		+\alpha\beta+\gamma
		\right)
		-\delta(\beta+1)+1}
	\,d(u,v),
\end{eqnarray*}
where
\[
\vartheta
=
\int_{0}^{1}
(1-h)^{-\frac12
	\left(
	\frac{n(\beta+1)}{\kappa}
	-\frac{n}{q}
	+\alpha\beta+\gamma
	\right)}
h^{-\delta(\beta+1)}
\,dh.
\]

Since
\[
\delta=\frac{n}{2q}-\frac{n}{2\kappa},
\]
we have
\[
-\frac12\left(
\frac{n(\beta+1)}{\kappa}
-\frac{n}{q}
+\alpha\beta+\gamma
\right)
-\delta(\beta+1)+1
=
1-\frac12\left(
\frac{n\beta}{q}
+\alpha\beta+\gamma
\right).
\]
Hence,
\[
\eqref{int1}
\lesssim
2cM^{\beta}\vartheta\,
t^{1-\frac12\left(\frac{n\beta}{q}+\alpha\beta+\gamma\right)}
d(u,v).
\]

Since
\[
q>\frac{n\beta}{2-\alpha\beta-\gamma},
\]
it follows that
\[
1-\frac12\left(\frac{n\beta}{q}+\alpha\beta+\gamma\right)>0.
\]
Moreover,
\[
\delta(\beta+1)<1,
\]
because
\[
\frac1q-\frac{2}{n(\beta+1)}
<
\frac1\kappa,
\]
by \eqref{assu-result31}. Furthermore,
\begin{eqnarray*}
	\frac12\left(
	\frac{n(\beta+1)}{\kappa}
	-\frac{n}{q}
	+\alpha\beta+\gamma
	\right)
	&<&
	\frac12\left(
	\frac{n(\beta+1)}{\kappa}
	-\frac{n}{\kappa}
	+\alpha\beta+\gamma
	\right)
	\\
	&=&
	\frac12\left(
	\frac{n\beta}{\kappa}
	+\alpha\beta+\gamma
	\right)
	<1,
\end{eqnarray*}
since
\[
\kappa>q>\frac{n\beta}{2-\alpha\beta-\gamma}.
\]
Therefore,
\begin{equation}
\|\Lambda(u)(t)-\Lambda(v)(t)\|_{\dot K_q^{\alpha,r}}
\lesssim
\|u_0-v_0\|_{\dot K_q^{\alpha,r}}
+
c_1M^{\beta}
T^{1-\frac12\left(\frac{n\beta}{q}+\alpha\beta+\gamma\right)}
\,d(u,v),
\label{main-result31}
\end{equation}
where the constant $c_1>0$ absorbs $2|a|c\vartheta$.
Using Lemma \ref{Bernstein-Herz-ine1} once again, we obtain
\begin{eqnarray*}
	\big\|e^{t\Delta }(u_{0}-v_{0})\big\|_{\dot{K}_{\kappa }^{\alpha ,r}}
	&\lesssim&
	t^{-\frac12\left(\frac{n}{q}-\frac{n}{\kappa}\right)}
	\big\|u_{0}-v_{0}\big\|_{\dot{K}_{q}^{\alpha ,r}}
	\\
	&=&
	ct^{-\delta}
	\big\|u_{0}-v_{0}\big\|_{\dot{K}_{q}^{\alpha ,r}},
\end{eqnarray*}
where $\delta=\frac{n}{2q}-\frac{n}{2\kappa}$, since
$\kappa>q$, $-\frac{n}{q}<\alpha<n-\frac{n}{q}$, and
$\alpha>-\frac{n}{\kappa}$.

Next,
\[
\big\|\Lambda(u)(t)-\Lambda(v)(t)\big\|_{\dot{K}_{\kappa}^{\alpha,r}}
\]
is bounded by
\begin{equation}
|a|
\int_{0}^{t}
\big\|
e^{(t-\tau)\Delta}
\big(
|\cdot|^{-\gamma}
(|u(\tau)|^{\beta}u(\tau)-|v(\tau)|^{\beta}v(\tau))
\big)
\big\|_{\dot{K}_{\kappa}^{\alpha,r}}
\,d\tau .
\label{int2}
\end{equation}

Applying Lemma \ref{Bernstein-Herz-ine1} with
\[
(\alpha_1,\alpha_2,p,q,s,\theta)
=
\left(
\alpha,
\alpha(\beta+1),
\kappa,
\frac{\kappa}{\beta+1},
\frac{r}{\beta+1},
\frac{r}{\beta+1}
\right),
\]
we obtain
\begin{eqnarray*}
	&&
	\big\|
	e^{(t-\tau)\Delta}
	\big(
	|\cdot|^{-\gamma}
	(|u(\tau)|^{\beta}u(\tau)-|v(\tau)|^{\beta}v(\tau))
	\big)
	\big\|_{\dot{K}_{\kappa}^{\alpha,r}}
	\\
	&\lesssim&
	(t-\tau)^{-\frac12\left(\frac{n\beta}{\kappa}+\alpha\beta+\gamma\right)}
	\big\|
	|u(\tau)|^{\beta}u(\tau)-|v(\tau)|^{\beta}v(\tau)
	\big\|_{\dot{K}_{\frac{\kappa}{\beta+1}}^{\alpha(\beta+1),\frac{r}{\beta+1}}}
	\\
	&\lesssim&
	(t-\tau)^{-\frac12\left(\frac{n\beta}{\kappa}+\alpha\beta+\gamma\right)}
	\big\|
	|u|^{\beta}+|v|^{\beta}
	\big\|_{\dot{K}_{\frac{\kappa}{\beta}}^{\alpha\beta,\frac{r}{\beta}}}
	\|u-v\|_{\dot{K}_{\kappa}^{\alpha,r}}
	\\
	&\le&
	2cM^{\beta}
	(t-\tau)^{-\frac12\left(\frac{n\beta}{\kappa}+\alpha\beta+\gamma\right)}
	\tau^{-\delta(\beta+1)}
	d(u,v).
\end{eqnarray*}

Consequently, \eqref{int2} is bounded by
\begin{eqnarray*}
	&&
	2|a|cM^{\beta}
	\int_{0}^{t}
	(t-\tau)^{-\frac12\left(\frac{n\beta}{\kappa}+\alpha\beta+\gamma\right)}
	\tau^{-\delta(\beta+1)}
	\,d\tau\,
	d(u,v)
	\\
	&=&
	2|a|cM^{\beta}\mu\,
	t^{-\delta}
	t^{-\frac12\left(\frac{n\beta}{\kappa}+\alpha\beta+\gamma\right)
		-\delta\beta+1}
	d(u,v),
\end{eqnarray*}
where
\[
\mu
=
\int_{0}^{1}
(1-h)^{-\frac12\left(\frac{n\beta}{\kappa}+\alpha\beta+\gamma\right)}
h^{-\delta(\beta+1)}
\,dh.
\]

Since
\[
\delta=\frac{n}{2q}-\frac{n}{2\kappa},
\]
we have
\[
-\frac12\left(\frac{n\beta}{\kappa}+\alpha\beta+\gamma\right)
-\delta\beta+1
=
1-\frac12\left(\frac{n\beta}{q}+\alpha\beta+\gamma\right).
\]
Therefore,
\[
\eqref{int2}
\lesssim
2|a|cM^{\beta}\mu\,
t^{-\delta}
t^{1-\frac12\left(\frac{n\beta}{q}+\alpha\beta+\gamma\right)}
d(u,v).
\]
Since
\[
q>\frac{n\beta}{2-\alpha\beta-\gamma},
\]
we have
\[
1-\frac12\left(\frac{n\beta}{q}+\alpha\beta+\gamma\right)>0.
\]
Moreover, since $\delta=\frac{n}{2q}-\frac{n}{2\kappa}$ with $\kappa>q$, we have
$\delta(\beta+1)<1$. Furthermore, $\kappa>q>\frac{n\beta}{2-\alpha\beta-\gamma}$
implies
\[
\frac12\left(\frac{n\beta}{\kappa}+\alpha\beta+\gamma\right)<1.
\]
Therefore,
\begin{equation}
t^{\delta}\|\Lambda(u)(t)-\Lambda(v)(t)\|_{\dot{K}_{q}^{\alpha,r}}
\lesssim
\|u_{0}-v_{0}\|_{\dot{K}_{q}^{\alpha,r}}
+c_{1}M^{\beta}
T^{1-\frac12(\frac{n\beta}{q}+\alpha\beta+\gamma)}
d(u,v).
\label{main-result32}
\end{equation}
Combining \eqref{main-result31} and \eqref{main-result32}, we obtain
\begin{equation}
d(\Lambda(u),\Lambda(v))
\lesssim
\|u_{0}-v_{0}\|_{\dot{K}_{q}^{\alpha,r}}
+c_{1}M^{\beta}
T^{1-\frac12(\frac{n\beta}{q}+\alpha\beta+\gamma)}
d(u,v).
\label{main-result33}
\end{equation}

Now let $u\in F_{M}$. Since
\[
1-\frac12\left(\frac{n\beta}{q}+\alpha\beta+\gamma\right)>0,
\]
it follows from the previous estimates that
\[
\Lambda(u)\in
C([0,T],\dot{K}_{q}^{\alpha,r})
\cap
C([0,T],\dot{K}_{\kappa}^{\alpha,r}).
\]
Moreover, by \eqref{assu1}, \eqref{assu2}, and \eqref{main-result33},
\begin{align*}
d(\Lambda(u),0)
&\lesssim
\|u_{0}\|_{\dot{K}_{q}^{\alpha,r}}
+c_{1}M^{\beta}
T^{1-\frac12(\frac{n\beta}{q}+\alpha\beta+\gamma)}
d(u,0) \\
&\le
\sigma
+c_{1}M^{\beta+1}
T^{1-\frac12(\frac{n\beta}{q}+\alpha\beta+\gamma)} \\
&\le M,
\end{align*}
which shows that $\Lambda$ maps $F_{M}$ into itself.

Taking $u_{0}=v_{0}$ in \eqref{main-result33}, we obtain
\[
d(\Lambda(u),\Lambda(v))
\le
c_{2}M^{\beta}
T^{1-\frac12(\frac{n\beta}{q}+\alpha\beta+\gamma)}
d(u,v).
\]
Hence, by \eqref{assu2}, $\Lambda$ is a strict contraction on $F_{M}$.
Therefore, the Banach fixed-point theorem guarantees the existence of a unique
fixed point of $\Lambda$ in $F_{M}$, which is the unique mild solution of
\eqref{int-equa1}.

To prove uniqueness for arbitrary $M$, it suffices to choose $T>0$
sufficiently small so that \eqref{assu2} remains valid. Finally, by a
standard continuation argument, this local solution extends to a maximal
solution. This completes the proof.
\end{proof}

In the next we study Theorem \ref{result3} when $\kappa \geq q(\beta +1)$.

\begin{theorem}
	\label{result3 copy(3)}
	Let $n\geq1$ be an integer, $1\leq q<\infty$, $1\leq r<\infty$, $\beta>0$ and
	$\gamma\in\mathbb{R}$ satisfy $\gamma<n$. Let $\kappa\geq1$ be such that
	\[
	\max\Big(\frac1q-\frac{2}{n(\beta+1)},\,
	\frac1q-\frac{2-\gamma}{n\beta}\Big)
	<
	\frac1\kappa
	<
	\min\Big(
	\frac1{q(\beta+1)},\,
	\frac{n-\gamma}{n(\beta+1)}+\frac{\gamma}{n\beta},\,
	\frac1{\beta+1}+\frac{\gamma}{n\beta}
	\Big).
	\]
	Assume that $\alpha\in\mathbb{R}$ satisfies
	\[
	\alpha\leq\gamma-\frac{n}{q}
	\]
	and
	\[
	-\frac{n}{\kappa}
	<
	\alpha
	<
	\min\Big(
	\frac{n-\gamma}{\beta+1}-\frac{n}{\kappa},\,
	\frac{n}{\beta+1}-\frac{n}{\kappa},\,
	\frac{2-\gamma}{\beta}-\frac{n}{q}
	\Big).
	\]
	Then \eqref{int-equa1} is locally well posed in $\dot{K}_{q}^{\alpha,r}$ as in
	Theorem~\ref{result1}, except that uniqueness holds only in the class of
	continuous functions $u$ satisfying
	\[
	\sup_{0<t\leq T}
	t^{\frac{n}{2q}-\frac{n}{2\kappa}}
	\|u(t)\|_{\dot{K}_{\kappa}^{\alpha,r}}
	<\infty.
	\]
	Moreover, the corresponding solution extends to a maximal interval
	$[0,T_{\max})$, where either $T_{\max}=\infty$, or $T_{\max}<\infty$ and
	\[
	\lim_{t\to T_{\max}}
	\|u(t)\|_{\dot{K}_{q}^{\alpha,r}}
	=\infty.
	\]
	Furthermore, the minimal existence time $T$ depends only on the norm of the
	initial datum $\|u_{0}\|_{\dot{K}_{q}^{\alpha,r}}$.
\end{theorem}

\begin{proof}
	The proof is analogous to that of Theorem~\ref{result3}. The only difference
	is that Lemma~\ref{Bernstein-Herz-ine2} is used in place of
	Lemma~\ref{Bernstein-Herz-ine1}. We therefore omit the details.
\end{proof}

\begin{remark}
	In view of Theorem~\ref{result3 copy(3)} and Lemma~\ref{Bernstein-Herz-ine2},
	the assumption
	\[
	\alpha \leq \gamma-\frac{n}{q}
	\]
	in Theorem~\ref{result3 copy(3)} can be relaxed. In particular, the case
	\[
	\alpha>\gamma-\frac{n}{q}
	\]
	can also be treated.
\end{remark}

We now study the limiting case $\alpha =\frac{n-\gamma }{\beta +1}-\frac{n}{\kappa 
}$\ of Theorem \ref{result3}.

\begin{theorem}
	\label{result3 copy(1)}
	Let $n\geq 1$ be an integer, $1<q<\infty$, $1\leq r<\infty$, $\beta>0$, and
	$\gamma$ satisfy
	\[
	0<\gamma<\min\bigl(2,n,n-(n-2)(\beta+1)\bigr).
	\]
	Let $\kappa>q$ be such that
	\[
	\frac{1}{\kappa}>
	\max\Bigg(
	\frac{n-\gamma}{n(\beta+1)}+\frac{1}{q}-1,\,
	\frac{n-\gamma}{n(\beta+1)}+\frac{1}{q}+\frac{\gamma-2}{n\beta},\,
	\frac{1}{q(\beta+1)},\,
	\frac{1}{q}-\frac{2}{n(\beta+1)}
	\Bigg)
	\]
	and
	\[
	\frac{1}{\kappa}
	<
	\min\Bigg(
	\frac{n-\gamma}{n(\beta+1)}+\frac{1}{q},\,
	\frac{1}{\beta}-\frac{n-\gamma}{n\beta(\beta+1)}
	\Bigg).
	\]
	Then \eqref{int-equa1} is locally well-posed in
	\[
	\dot{K}_{q}^{\frac{n-\gamma}{\beta+1}-\frac{n}{\kappa},1},
	\]
	as in Theorem \ref{result1}, except that uniqueness holds only in the class
	of continuous functions $u$ satisfying
	\[
	t^{\frac{n}{2q}-\frac{n}{2\kappa}}
	\|u(t)\|_{\dot{K}_{\kappa}^{\frac{n-\gamma}{\beta+1}-\frac{n}{\kappa},1}}
	\]
	is bounded on $(0,T]$.
	
	Moreover, the solution extends to a maximal interval of existence
	$[0,T_{\max})$ such that either $T_{\max}=\infty$, or
	$T_{\max}<\infty$ and
	\[
	\lim_{t\to T_{\max}}
	\|u(t)\|_{\dot{K}_{q}^{\frac{n-\gamma}{\beta+1}-\frac{n}{\kappa},1}}
	=\infty.
	\]
	Furthermore, the minimal existence time $T$ depends only on
	\[
	\|u_{0}\|_{\dot{K}_{q}^{\frac{n-\gamma}{\beta+1}-\frac{n}{\kappa},1}}.
	\]
\end{theorem}

\begin{proof}
Recall that
\[
\delta =\frac{n}{2q}-\frac{n}{2\kappa }.
\]
Let $M,\sigma,T>0$ and
\[
u_{0}\in \dot{K}_{q}^{\frac{n-\gamma}{\beta+1}-\frac{n}{\kappa},1}
\]
satisfy
\[
\|u_{0}\|_{\dot{K}_{q}^{\frac{n-\gamma}{\beta+1}-\frac{n}{\kappa},1}}
\leq M,
\qquad
cM^{\beta}
T^{1-\frac12\left(\frac{n\beta}{q}
	+\left(\frac{n-\gamma}{\beta+1}-\frac{n}{\kappa}\right)\beta+\gamma\right)}
<1,
\]
and
\[
cM^{\beta+1}
T^{1-\frac12\left(\frac{n\beta}{q}
	+\left(\frac{n-\gamma}{\beta+1}-\frac{n}{\kappa}\right)\beta+\gamma\right)}
\leq M,
\]
where $c>0$ is a constant.

The proof relies on the Banach contraction mapping principle. Define
$H_{M}$ as the set of all functions
\[
u\in C([0,T],
\dot{K}_{q}^{\frac{n-\gamma}{\beta+1}-\frac{n}{\kappa},1})
\cap
C([0,T],
\dot{K}_{\kappa}^{\frac{n-\gamma}{\beta+1}-\frac{n}{\kappa},1})
\]
such that
\[
\max\left\{
\sup_{t\in[0,T]}
\|u(t)\|_{\dot{K}_{q}^{\frac{n-\gamma}{\beta+1}-\frac{n}{\kappa},1}},
\,
\sup_{t\in[0,T]}
t^{\delta}
\|u(t)\|_{\dot{K}_{\kappa}^{\frac{n-\gamma}{\beta+1}-\frac{n}{\kappa},1}}
\right\}
\leq M.
\]

We endow $H_{M}$ with the metric
\[
d(u,v)
=
\max\left\{
\sup_{t\in[0,T]}
\|u(t)-v(t)\|_{\dot{K}_{q}^{\frac{n-\gamma}{\beta+1}-\frac{n}{\kappa},1}},
\,
\sup_{t\in[0,T]}
t^{\delta}
\|u(t)-v(t)\|_{\dot{K}_{\kappa}^{\frac{n-\gamma}{\beta+1}-\frac{n}{\kappa},1}}
\right\}.
\]

We shall prove that the integral equation \eqref{int-equa1} admits a unique
solution $u\in H_{M}$. To this end, let
\[
u_{0},v_{0}\in
\dot{K}_{q}^{\frac{n-\gamma}{\beta+1}-\frac{n}{\kappa},1},
\qquad
u,v\in H_{M}.
\]
By Lemma \ref{Bernstein-Herz-ine1}, we obtain
\begin{equation*}
\big\|e^{t\Delta}(u_{0}-v_{0})\big\|_{\dot{K}_{q}^{\frac{n-\gamma}{\beta+1}-\frac{n}{\kappa},1}}
\lesssim
\big\|u_{0}-v_{0}\big\|_{\dot{K}_{q}^{\frac{n-\gamma}{\beta+1}-\frac{n}{\kappa},1}},
\end{equation*}
since
\[
-\frac{n}{q}
<
\frac{n-\gamma}{\beta+1}-\frac{n}{\kappa}
<
n-\frac{n}{q}.
\]
As before, we define
\begin{equation*}
\Lambda(u)(t)
=
|a|
\int_{0}^{t}
e^{(t-\tau)\Delta}
\bigl(|\cdot|^{-\gamma}|u(\tau)|^{\beta}u(\tau)\bigr)
\,d\tau.
\end{equation*}
Then
\[
\big\|\Lambda(u)(t)-\Lambda(v)(t)\big\|_{\dot{K}_{q}^{\frac{n-\gamma}{\beta+1}-\frac{n}{\kappa},1}}
\]
is bounded by
\begin{equation}
|a|
\int_{0}^{t}
\Big\|
e^{(t-\tau)\Delta}
\bigl(
|\cdot|^{-\gamma}
\bigl(|u(\tau)|^{\beta}u(\tau)-|v(\tau)|^{\beta}v(\tau)\bigr)
\bigr)
\Big\|_{\dot{K}_{q}^{\frac{n-\gamma}{\beta+1}-\frac{n}{\kappa},1}}
\,d\tau .
\label{int3}
\end{equation}

Applying Lemma \ref{Bernstein-Herz-ine1 copy(1)}, we obtain
\begin{align*}
&\Big\|
e^{(t-\tau)\Delta}
\bigl(
|\cdot|^{-\gamma}
\bigl(|u(\tau)|^{\beta}u(\tau)-|v(\tau)|^{\beta}v(\tau)\bigr)
\bigr)
\Big\|_{\dot{K}_{q}^{\frac{n-\gamma}{\beta+1}-\frac{n}{\kappa},1}}
\\
&\qquad\lesssim
(t-\tau)^{-\frac12\left(
	n+\frac{n}{\kappa}-\frac{n}{q}-\frac{n-\gamma}{\beta+1}
	\right)}
\Big\|
|u(\tau)|^{\beta}u(\tau)-|v(\tau)|^{\beta}v(\tau)
\Big\|_{\dot{K}_{\frac{\kappa}{\beta+1}}^{\left(\frac{n-\gamma}{\beta+1}-\frac{n}{\kappa}\right)(\beta+1),\,\frac{1}{\beta+1}}}
\\
&\qquad\lesssim
(t-\tau)^{-\frac12\left(
	n+\frac{n}{\kappa}-\frac{n}{q}-\frac{n-\gamma}{\beta+1}
	\right)}
\big\||u|^{\beta}+|v|^{\beta}\big\|_{\dot{K}_{\frac{\kappa}{\beta}}^{\left(\frac{n-\gamma}{\beta+1}-\frac{n}{\kappa}\right)\beta,\frac{1}{\beta}}}
\|u-v\|_{\dot{K}_{\kappa}^{\frac{n-\gamma}{\beta+1}-\frac{n}{\kappa},1}}
\\
&\qquad\le
2cM^{\beta}
(t-\tau)^{-\frac12\left(
	n+\frac{n}{\kappa}-\frac{n}{q}-\frac{n-\gamma}{\beta+1}
	\right)}
\tau^{-\delta(\beta+1)}
d(u,v),
\end{align*}
since
\[
q>\frac{\kappa}{\beta+1},
\qquad
\frac{n-\gamma}{\beta+1}-\frac{n}{\kappa}>-\frac{n}{q},
\]
and
\begin{align*}
\max\left(
\frac{n-\gamma}{\beta+1}-\frac{n}{\kappa}-\gamma,
-\frac{n(\beta+1)}{\kappa}
\right)
&<
\left(
\frac{n-\gamma}{\beta+1}-\frac{n}{\kappa}
\right)(\beta+1)
\\
&=
n-\gamma-\frac{n(\beta+1)}{\kappa}
<
n-\frac{n(\beta+1)}{\kappa}.
\end{align*}
Hence, \eqref{int3} is bounded by
\begin{eqnarray*}
	&&2|a|cM^{\beta }\int_{0}^{t}(t-\tau )^{-\frac{1}{2}\left(n+\frac{n}{\kappa }-
		\frac{n}{q}-\frac{n-\gamma }{\beta +1}\right)}
	\tau ^{-\delta (\beta +1)}\,d\tau \, d(u,v) \\
	&=&2|a|c\Upsilon M^{\beta }
	t^{-\frac{1}{2}\left(n+\frac{n}{\kappa }-\frac{n}{q}
		-\frac{n-\gamma }{\beta +1}\right)-\delta (\beta +1)+1}
	d(u,v),
\end{eqnarray*}
where
\begin{equation*}
\Upsilon
=\int_{0}^{1}
(1-h)^{-\frac{1}{2}\left(n+\frac{n}{\kappa }-\frac{n}{q}
	-\frac{n-\gamma }{\beta +1}\right)}
h^{-\delta (\beta +1)}\,dh.
\end{equation*}
Since $\delta=\frac{n}{2q}-\frac{n}{2\kappa }$, we have
\begin{equation*}
-\frac{1}{2}\left(
n+\frac{n}{\kappa }-\frac{n}{q}
-\frac{n-\gamma }{\beta +1}
\right)
-\delta(\beta+1)+1
=
-\frac{1}{2}\left(
n+\frac{n\beta}{q}
-\frac{n\beta}{\kappa}
-\frac{n-\gamma}{\beta+1}
\right)+1.
\end{equation*}
Therefore,
\begin{equation*}
\eqref{int3}
\lesssim
2c\Upsilon M^{\beta }
t^{1-\frac{1}{2}\left(
	n+\frac{n\beta}{q}
	-\frac{n\beta}{\kappa}
	-\frac{n-\gamma}{\beta+1}
	\right)}
d(u,v).
\end{equation*}

Since
\[
q>
\frac{n\beta}
{2-\left(\frac{n-\gamma}{\beta+1}-\frac{n}{\kappa}\right)\beta-\gamma},
\]
it follows that
\begin{equation*}
1-\frac{1}{2}\left(
n+\frac{n\beta}{q}
-\frac{n\beta}{\kappa}
-\frac{n-\gamma}{\beta+1}
\right)>0.
\end{equation*}
Moreover,
\[
\delta(\beta+1)<1,
\]
since
\begin{equation*}
\frac{1}{q}-\frac{2}{n(\beta+1)}
<
\frac{1}{\kappa }.
\end{equation*}
Furthermore, since $\kappa>q$ and
\[
0<\gamma<\min\bigl(2,n,n-(n-2)(\beta+1)\bigr),
\]
we have
\begin{equation*}
n+\frac{n}{\kappa }-\frac{n}{q}
-\frac{n-\gamma}{\beta+1}
<
n-\frac{n-\gamma}{\beta+1}
<2.
\end{equation*}
Hence, $\Upsilon<\infty$. Consequently,
\begin{equation*}
\big\|\Lambda_t(u)-\Lambda_t(v)\big\|_
{\dot{K}_{q}^{\frac{n-\gamma}{\beta+1}-\frac{n}{\kappa},1}}
\lesssim
\big\|u_{0}-v_{0}\big\|_
{\dot{K}_{q}^{\frac{n-\gamma}{\beta+1}-\frac{n}{\kappa},1}}
+c_{1}M^{\beta}
T^{1-\frac12\left(
	\frac{n\beta}{q}
	+\left(\frac{n-\gamma}{\beta+1}-\frac{n}{\kappa}\right)\beta
	+\gamma\right)}
d(u,v).
\end{equation*}

By Lemma~\ref{Bernstein-Herz-ine1},
\begin{eqnarray*}
	\big\|e^{t\Delta}(u_{0}-v_{0})\big\|_
	{\dot{K}_{\kappa}^{\frac{n-\gamma}{\beta+1}-\frac{n}{\kappa},1}}
	&\lesssim&
	t^{-\frac12\left(\frac{n}{q}-\frac{n}{\kappa}\right)}
	\big\|u_{0}-v_{0}\big\|_
	{\dot{K}_{q}^{\frac{n-\gamma}{\beta+1}-\frac{n}{\kappa},1}}\\
	&=&
	ct^{-\delta}
	\big\|u_{0}-v_{0}\big\|_
	{\dot{K}_{q}^{\frac{n-\gamma}{\beta+1}-\frac{n}{\kappa},1}},
\end{eqnarray*}
since $\kappa>q$,
\[
-\frac{n}{q}
<
\frac{n-\gamma}{\beta+1}-\frac{n}{\kappa}
<
n-\frac{n}{q},
\]
and
\[
\frac{n-\gamma}{\beta+1}-\frac{n}{\kappa}
>
-\frac{n}{\kappa}.
\]
It remains to estimate
\begin{equation*}
\big\|\Lambda_t(u)-\Lambda_t(v)\big\|_
{\dot{K}_{\kappa}^{\frac{n-\gamma}{\beta+1}-\frac{n}{\kappa},1}}.
\end{equation*}
The nonlinear term is bounded by
\begin{equation}
|a|
\int_{0}^{t}
\big\|
e^{(t-\tau)\Delta}
\big(
|\cdot|^{-\gamma}
(
|u(\tau)|^{\beta}u(\tau)
-
|v(\tau)|^{\beta}v(\tau)
)
\big)
\big\|_
{\dot{K}_{\kappa}^{\frac{n-\gamma}{\beta+1}-\frac{n}{\kappa},1}}
\,d\tau.
\label{int4}
\end{equation}
Now, by Lemma \ref{Bernstein-Herz-ine1 copy(1)}, we obtain
\begin{eqnarray*}
	&&\big\|e^{(t-\tau )\Delta }\big(|\cdot |^{-\gamma }
	\big(|u(\tau )|^{\beta }u(\tau )
	-|v(\tau )|^{\beta }v(\tau )\big)\big)\big\|
	_{\dot{K}_{\kappa }^{\frac{n-\gamma }{\beta +1}-\frac{n}{\kappa },1}} \\
	&\lesssim&
	(t-\tau )^{-\frac12\left(n-\frac{n-\gamma}{\beta+1}\right)}
	\big\||u(\tau )|^{\beta }u(\tau )
	-|v(\tau )|^{\beta }v(\tau )\big\|
	_{\dot{K}_{\frac{\kappa}{\beta+1}}^{
			\left(\frac{n-\gamma}{\beta+1}-\frac{n}{\kappa}\right)(\beta+1),1}} \\
	&\lesssim&
	(t-\tau )^{-\frac12\left(n-\frac{n-\gamma}{\beta+1}\right)}
	\big\||u|^{\beta }+|v|^{\beta }\big\|
	_{\dot{K}_{\frac{\kappa}{\beta}}^{\alpha\beta,\frac{r}{\beta}}}
	\|u-v\|_{\dot{K}_{\kappa}^{\frac{n-\gamma}{\beta+1}-\frac{n}{\kappa},1}} \\
	&\le&
	2cM^{\beta }
	(t-\tau )^{-\frac12\left(n-\frac{n-\gamma}{\beta+1}\right)}
	\tau^{-\delta(\beta+1)}
	d(u,v),
\end{eqnarray*}
because
\[
\kappa>\frac{\kappa}{\beta+1},
\qquad
\frac{n-\gamma}{\beta+1}-\frac{n}{\kappa}
>-\frac{n}{\kappa},
\]
and
\begin{eqnarray*}
	\max\left(
	\frac{n-\gamma}{\beta+1}-\frac{n}{\kappa}-\gamma,
	-\frac{n(\beta+1)}{\kappa}
	\right)
	&\le&
	\left(
	\frac{n-\gamma}{\beta+1}
	-\frac{n}{\kappa}
	\right)(\beta+1) \\
	&=&
	n-\frac{n(\beta+1)}{\kappa}-\gamma \\
	&<&
	n-\frac{n(\beta+1)}{\kappa}.
\end{eqnarray*}

Hence, \eqref{int4} is bounded by
\[
2|a|cM^{\beta}
\int_0^t
(t-\tau)^{-\frac12\left(n-\frac{n-\gamma}{\beta+1}\right)}
\tau^{-\delta(\beta+1)}
\,d\tau\,
d(u,v)
=
2|a|c\lambda M^{\beta}
t^{-\delta}
t^{1-\frac12\left(n-\frac{n-\gamma}{\beta+1}\right)-\delta\beta}
d(u,v),
\]
where
\[
\lambda
=
\int_0^1
(1-h)^{-\frac12\left(n-\frac{n-\gamma}{\beta+1}\right)}
h^{-\delta(\beta+1)}
\,dh.
\]

Since
\[
\delta=\frac{n}{2q}-\frac{n}{2\kappa},
\]
we have
\[
1-\frac12\left(n-\frac{n-\gamma}{\beta+1}\right)-\delta\beta
=
1-\frac12\left(
n-\frac{n-\gamma}{\beta+1}
+\frac{n\beta}{q}
-\frac{n\beta}{\kappa}
\right).
\]
Therefore,
\[
\eqref{int4}
\lesssim
2c\lambda M^{\beta}
t^{-\delta}
t^{1-\frac12\left(
	n-\frac{n-\gamma}{\beta+1}
	+\frac{n\beta}{q}
	-\frac{n\beta}{\kappa}
	\right)}
d(u,v).
\]

Since
\[
q>
\frac{n\beta}
{2-
	\left(
	\frac{n-\gamma}{\beta+1}
	-\frac{n}{\kappa}
	\right)\beta
	-\gamma},
\]
it follows that
\[
1-\frac12\left(
n-\frac{n-\gamma}{\beta+1}
+\frac{n\beta}{q}
-\frac{n\beta}{\kappa}
\right)
>0.
\]

Moreover,
\[
\delta(\beta+1)<1.
\]
Since
\[
\kappa>q>
\frac{n\beta}
{2-
	\left(
	\frac{n-\gamma}{\beta+1}
	-\frac{n}{\kappa}
	\right)\beta
	-\gamma},
\]
we also have
\[
\frac12\left(
n-\frac{n-\gamma}{\beta+1}
\right)<1.
\]

Consequently,
\[
t^\delta
\|\Lambda_t(u)-\Lambda_t(v)\|_
{\dot{K}_q^{\frac{n-\gamma}{\beta+1}-\frac{n}{\kappa},1}}
\lesssim
\|u_0-v_0\|_
{\dot{K}_q^{\frac{n-\gamma}{\beta+1}-\frac{n}{\kappa},1}}
+
c_1M^\beta
T^{1-\frac12\left(
	n-\frac{n-\gamma}{\beta+1}
	+\frac{n\beta}{q}
	-\frac{n\beta}{\kappa}
	\right)}
d(u,v).
\]

The remainder of the proof follows exactly as in the proof of Theorem
\ref{result3}.
\end{proof}

As in the proof of Theorem \ref{result3 copy(1)}, using Lemma
\ref{Bernstein-Herz-ine2 copy(1)}, we obtain the following result.
\begin{theorem}
	\label{result3 copy(2)}
	Let $n\geq 1$ be an integer, $1<q<\infty$, $1\leq r<\infty$, $\beta>0$, and
	$\gamma$ satisfy
	\[
	0<\gamma<\min\bigl(2,n,n-(n-2)(\beta+1)\bigr).
	\]
	Let $\kappa\geq q(\beta+1)$ be such that
	\[
	\frac{1}{\kappa}
	>
	\max\left(
	\frac{n-\gamma}{n(\beta+1)}+\frac{1}{q}-1,\
	\frac{n-\gamma}{n(\beta+1)}+\frac{1}{q}+\frac{\gamma-2}{n\beta},\
	\frac{1}{q}-\frac{2}{n(\beta+1)}
	\right)
	\]
	and
	\[
	\frac{1}{\kappa}
	<
	\min\left(
	\frac{n-\gamma}{n(\beta+1)}+\frac{1}{q},\
	\frac{1}{\beta}-\frac{n-\gamma}{n\beta(\beta+1)}
	\right).
	\]
	Then the integral equation \eqref{int-equa1} is locally well posed in
	$\dot{K}_{q}^{\frac{n-\gamma}{\beta+1}-\frac{n}{\kappa},1}$ as in Theorem
	\ref{result1}, except that uniqueness is guaranteed only in the class of
	continuous functions $u$ satisfying
	\[
	t^{\frac{n}{2q}-\frac{n}{2\kappa}}
	\|u(t)\|_{\dot{K}_{\kappa}^{\frac{n-\gamma}{\beta+1}-\frac{n}{\kappa},1}}
	\in L^\infty(0,T].
	\]
	Moreover, the solution $u$ can be extended to a maximal interval
	$[0,T_{\max})$ such that either $T_{\max}=\infty$, or
	$T_{\max}<\infty$ and
	\[
	\lim_{t\to T_{\max}}
	\|u(t)\|_{\dot{K}_{q}^{\frac{n-\gamma}{\beta+1}-\frac{n}{\kappa},1}}
	=\infty.
	\]
	Furthermore, the minimal existence time $T$ depends only on
	\[
	\|u_{0}\|_{\dot{K}_{q}^{\frac{n-\gamma}{\beta+1}-\frac{n}{\kappa},1}}.
	\]
\end{theorem}

\section{Global existence}

In this section, we study the global well-posedness of \eqref{int-equa1}. To
this end, we first establish the following lemma.

\begin{lemma}
\label{keylemma1}Let $\beta >0,\gamma \in \mathbb{R},n\in \mathbb{N}$ and $%
s,q$ be two real numbers. Suppose that $\beta +1<s<q\leq \infty ,\alpha >-%
\frac{n}{q},-\beta \alpha \leq \gamma <n$,%
\begin{equation*}
\frac{1}{s}<\min \Big(\frac{n-\gamma }{n(\beta +1)}-\frac{\alpha }{n},\frac{n%
}{n(\beta +1)}-\frac{\alpha }{n}\Big)
\end{equation*}%
and%
\begin{equation*}
\frac{n}{2}\Big(\frac{\beta +1}{s}-\frac{1}{q}\Big)<1-\frac{\alpha \beta
+\gamma }{2}.
\end{equation*}%
Let $u$ be the solution of \eqref{int-equa1} with initial data $u_{0}\in 
\mathcal{S}^{\prime }(\mathbb{R}^{n})$. Assume that%
\begin{equation*}
\sup_{t>0}t^{\frac{2-\alpha \beta -\gamma }{2\beta }-\frac{n}{2s}}\big\|u(t)%
\big\|_{\dot{K}_{s}^{\alpha ,r}}<\infty .
\end{equation*}%
Then it follows that 
\begin{equation}
\sup_{t>0}t^{\frac{2-\alpha \beta -\gamma }{2\beta }-\frac{n}{2q}}\big\|u(t)%
\big\|_{\dot{K}_{q}^{\alpha ,r}}<\infty .  \label{main4}
\end{equation}
\end{lemma}

\begin{proof}
	The proof is based on an argument similar to that in \cite{STW01}, together
	with Lemma \ref{Bernstein-Herz-ine1}. Set
	\[
	A=\sup_{t>0}t^{\delta(s)}
	\|u(t)\|_{\dot{K}_{s}^{\alpha,r}},
	\qquad
	\delta(s)=\frac{2-\alpha\beta-\gamma}{2\beta}-\frac{n}{2s}.
	\]
	
	Using the integral equation \eqref{int-equa1} on the interval
	$[\frac{t}{2},t]$, we write
	\[
	u(t)=e^{\frac{t}{2}\Delta}u\!\left(\frac{t}{2}\right)
	+a\int_{\frac{t}{2}}^{t}
	e^{(t-\tau)\Delta}
	\bigl(|\cdot|^{-\gamma}|u(\tau)|^{\beta}u(\tau)\bigr)\,d\tau.
	\]
	
	By Lemma \ref{Bernstein-Herz-ine1},
	\begin{eqnarray*}
		\left\|e^{\frac{t}{2}\Delta}u\!\left(\frac{t}{2}\right)\right\|
		_{\dot{K}_{q}^{\alpha,r}}
		&\lesssim&
		t^{-\frac12(\frac ns-\frac nq)}
		\left\|u\!\left(\frac{t}{2}\right)\right\|_{\dot{K}_{s}^{\alpha,r}} \\
		&\lesssim&
		t^{-\frac12(\frac ns-\frac nq)-\delta(s)}A \\
		&=&
		Ct^{\frac{n}{2q}-\frac{2-\alpha\beta-\gamma}{2\beta}}A.
	\end{eqnarray*}
	
	Moreover,
	\begin{eqnarray*}
		&&\left\|
		e^{(t-\tau)\Delta}
		\bigl(|\cdot|^{-\gamma}|u(\tau)|^{\beta}u(\tau)\bigr)
		\right\|_{\dot{K}_{q}^{\alpha,r}} \\
		&\lesssim&
		(t-\tau)^{-\frac12\left(
			\frac{n(\beta+1)}{s}
			-\frac nq
			+\alpha\beta+\gamma
			\right)}
		\bigl\||u(\tau)|^{\beta}u(\tau)\bigr\|
		_{\dot{K}_{\frac{s}{\beta+1}}^{\alpha(\beta+1),\,\frac{r}{\beta+1}}} \\
		&\lesssim&
		(t-\tau)^{-\frac12\left(
			\frac{n(\beta+1)}{s}
			-\frac nq
			+\alpha\beta+\gamma
			\right)}
		\|u(\tau)\|_{\dot{K}_{s}^{\alpha,r}}^{\beta+1},
	\end{eqnarray*}
	since $q>s$, $\alpha>-\frac nq$, and
	\[
	\max\left(
	\alpha-\gamma,
	-\frac{n(\beta+1)}{s}
	\right)
	\le
	\alpha(\beta+1)
	<
	\min\left(
	n-\frac{n(\beta+1)}{s}-\gamma,
	n-\frac{n(\beta+1)}{s}
	\right).
	\]
	
	Therefore,
	\[
	\int_{\frac{t}{2}}^{t}
	\left\|
	e^{(t-\tau)\Delta}
	\bigl(|\cdot|^{-\gamma}|u(\tau)|^{\beta}u(\tau)\bigr)
	\right\|_{\dot{K}_{q}^{\alpha,r}}
	\,d\tau
	\]
	is bounded above by
	\begin{eqnarray*}
		&&C
		\int_{\frac{t}{2}}^{t}
		(t-\tau)^{-\frac12\left(
			\frac{n(\beta+1)}{s}
			-\frac nq
			+\alpha\beta+\gamma
			\right)}
		\|u(\tau)\|_{\dot{K}_{s}^{\alpha,r}}^{\beta+1}
		\,d\tau \\
		&\lesssim&
		A^{\beta+1}
		\int_{\frac{t}{2}}^{t}
		(t-\tau)^{-\frac12\left(
			\frac{n(\beta+1)}{s}
			-\frac nq
			+\alpha\beta+\gamma
			\right)}
		\tau^{-\delta(s)(\beta+1)}
		\,d\tau \\
		&=&
		Ct^{-\frac12\left(
			\frac{n(\beta+1)}{s}
			-\frac nq
			+\alpha\beta+\gamma
			\right)-\delta(s)(\beta+1)+1}
		A^{\beta+1}
		\int_{\frac12}^{1}
		(1-h)^{-\frac12\left(
			\frac{n(\beta+1)}{s}
			-\frac nq
			+\alpha\beta+\gamma
			\right)}
		h^{-\delta(s)(\beta+1)}
		\,dh \\
		&=&
		C_1
		t^{-\frac12\left(
			\frac{n(\beta+1)}{s}
			-\frac nq
			+\alpha\beta+\gamma
			\right)-\delta(s)(\beta+1)+1}
		A^{\beta+1},
	\end{eqnarray*}
	where
	\[
	C_1
	=
	C
	\int_{\frac12}^{1}
	(1-h)^{-\frac12\left(
		\frac{n(\beta+1)}{s}
		-\frac nq
		+\alpha\beta+\gamma
		\right)}
	h^{-\delta(s)(\beta+1)}
	\,dh.
	\]
	The above integral is finite because
	\[
	\frac{n(\beta+1)}{s}
	-\frac nq
	+\alpha\beta+\gamma
	<2.
	\]
	
	Furthermore,
	\[
	-\frac12\left(
	\frac{n(\beta+1)}{s}
	-\frac nq
	+\alpha\beta+\gamma
	\right)
	-\delta(s)(\beta+1)
	+1
	=
	\frac{n}{2q}
	-\frac{2-\alpha\beta-\gamma}{2\beta}.
	\]
	Hence,
	\[
	t^{\frac{2-\alpha\beta-\gamma}{2\beta}-\frac{n}{2q}}
	\|u(t)\|_{\dot{K}_{q}^{\alpha,r}}
	\lesssim
	A+A^{\beta+1},
	\]
	for every $t>0$, where the implicit constant is independent of $t$. This proves \eqref{main4}.
\end{proof}

\begin{remark}
Since $\beta >0$, we have 
\begin{equation*}
A+A^{\beta +1}\rightarrow 0\quad \text{as }A\rightarrow 0.
\end{equation*}
\end{remark}

We are now ready to state the main result of this section.
\begin{theorem}
	\label{result5}
	Let $\beta>0$, $n\in\mathbb{N}$, $\alpha,\gamma\in\mathbb{R}$, and
	$\kappa>q_{c}>1$ be such that
	\begin{equation*}
	\max \left(
	-\frac{n}{\beta+1},
	-n+\frac{2-n}{\beta+1},
	-\frac{n}{\beta}+\frac{2-n\beta}{\beta(\beta+1)}
	\right)
	<
	\alpha
	<
	\min \left(
	n,\frac{2+n}{\beta+1}
	\right),
	\end{equation*}
	\begin{equation*}
	\max \left(
	-\alpha\beta,
	2-n\beta,
	2-n\beta-\alpha\beta,
	-\alpha\beta+\frac{2-n\beta}{\beta+1},
	\frac{2-n\beta}{\beta+1}
	\right)
	<
	\gamma
	<
	\min (2-\alpha\beta,n),
	\end{equation*}
	and
	\begin{equation}
	\max \left(
	\frac{1}{q_{c}}-\frac{2}{n(\beta+1)},
	-\frac{\alpha}{n}
	\right)
	<
	\frac{1}{\kappa}
	<
	\min \left(
	\frac{n-\gamma}{n(\beta+1)}-\frac{\alpha}{n},
	\frac{1}{\beta+1},
	\frac{1}{\beta+1}-\frac{\alpha}{n}
	\right).
	\label{assumption1}
	\end{equation}
	Let
	\[
	\delta=\frac{n}{2q_{c}}-\frac{n}{2\kappa}.
	\]
	Suppose that $\sigma>0$ and $M>0$ satisfy
	\[
	\sigma+VM^{\beta+1}\leq M,
	\]
	where $V=V(\alpha,\beta,n,\gamma,\kappa)>0$ is an explicitly computable
	constant. Let $u_{0}$ be a tempered distribution satisfying
	\begin{equation}
	\sup_{t>0}
	t^{\delta}
	\big\|e^{t\Delta}u_{0}\big\|_{\dot{K}_{\kappa}^{\alpha,r}}
	\leq
	\sigma.
	\label{assumption2}
	\end{equation}
	Then there exists a unique global solution $u$ of \eqref{int-equa1}
	satisfying
	\[
	\sup_{t>0}
	t^{\delta}
	\|u(t)\|_{\dot{K}_{\kappa}^{\alpha,r}}
	\leq
	M.
	\]
	
	Furthermore, the following properties hold:
	\begin{enumerate}
		\item[(i)]
		\[
		u-e^{t\Delta}u_{0}
		\in
		C([0,\infty);\dot{K}_{s}^{\alpha,r}),
		\qquad
		\max\left(\frac{1}{q_{c}},-\frac{\alpha}{n}\right)
		<
		\frac{1}{s}
		\leq
		\frac{\beta+1}{\kappa}.
		\]
		
		\item[(ii)]
		\[
		u-e^{t\Delta}u_{0}
		\in
		C((0,\infty);\dot{K}_{s}^{\alpha,r}),
		\]
		provided that
		\[
		\max\left(\frac{1}{q_{c}},-\frac{\alpha}{n}\right)
		\leq
		\frac{1}{s}
		\leq
		\frac{\beta+1}{\kappa},
		\qquad
		\frac{1}{s}\neq-\frac{\alpha}{n}.
		\]
		
		\item[(iii)]
		\[
		u(t)\longrightarrow u_{0}
		\quad\text{in }\mathcal{S}'(\mathbb{R}^{n})
		\quad\text{as }t\rightarrow0.
		\]
		
		\item[(iv)]
		For every $q\in[\kappa,\infty]$,
		\[
		\sup_{t>0}
		t^{\frac{2-\alpha\beta-\gamma}{2\beta}-\frac{n}{2q}}
		\|u(t)\|_{\dot{K}_{q}^{\alpha,r}}
		<\infty.
		\]
		In the case $q=\infty$, assume that $\alpha\ge0$.
	\end{enumerate}
	
	Moreover, let $u_{0}$ and $v_{0}$ satisfy \eqref{assumption2}, and let
	$u$ and $v$ denote the corresponding solutions of \eqref{int-equa1}.
	Then
	\begin{equation}
	\sup_{t>0}
	t^{\frac{2-\alpha\beta-\gamma}{2\beta}-\frac{n}{2q}}
	\|u(t)-v(t)\|_{\dot{K}_{\kappa}^{\alpha,r}}
	\leq
	c
	\sup_{t>0}
	t^{\delta}
	\|u(t)-v(t)\|_{\dot{K}_{\kappa}^{\alpha,r}},
	\label{estu-v1}
	\end{equation}
	for every $q\in[\kappa,\infty]$. In the case $q=\infty$, assume that
	$\alpha\ge0$.
	
	If, in addition,
	\[
	\sup_{t>0}
	t^{\eta+\delta}
	\big\|e^{t\Delta}(u_{0}-v_{0})\big\|_{\dot{K}_{\kappa}^{\alpha,r}}
	<\infty
	\]
	for some $\eta>0$ satisfying
	\[
	\delta(\beta+1)+\eta<1,
	\]
	possibly after reducing $M$, then
	\begin{equation}
	\sup_{t>0}
	t^{\eta+\delta}
	\|u(t)-v(t)\|_{\dot{K}_{\kappa}^{\alpha,r}}
	\leq
	c
	\sup_{t>0}
	t^{\eta+\delta}
	\big\|e^{t\Delta}(u_{0}-v_{0})\big\|_{\dot{K}_{\kappa}^{\alpha,r}},
	\label{estu-v2}
	\end{equation}
	where $c>0$ is a constant.
\end{theorem}

\begin{proof}
The proof is divided into five steps.

\textbf{Step 1.}
Let $X$ be the set of Bochner measurable functions
\[
u:(0,\infty)\rightarrow \dot{K}_{\kappa}^{\alpha,r}
\]
such that
\[
\sup_{t>0}t^{\delta}\|u(t)\|_{\dot{K}_{\kappa}^{\alpha,r}}<\infty.
\]
Let $X_M$ denote the subset of $X$ consisting of all functions satisfying
\[
\sup_{t>0}t^{\delta}\|u(t)\|_{\dot{K}_{\kappa}^{\alpha,r}}\le M.
\]
Equipped with the metric
\[
d(u,v)=\sup_{t>0}t^{\delta}
\|u(t)-v(t)\|_{\dot{K}_{\kappa}^{\alpha,r}},
\]
the set $X_M$ is a nonempty complete metric space.

Consider the mapping
\begin{equation}
\Im_{u_0}(u)(t)
=e^{t\Delta}u_0
+a\int_0^te^{(t-\tau)\Delta}
\big(|\cdot|^{-\gamma}|u(\tau)|^{\beta}u(\tau)\big)\,d\tau,
\label{int-equa2}
\end{equation}
where $u_0$ is a tempered distribution satisfying
\eqref{assumption2}. We shall prove that $\Im_{u_0}$ is a strict
contraction on $X_M$.

Let $u_0,v_0$ satisfy \eqref{assumption2} and let
$u,v\in X_M$. Then
\begin{eqnarray*}
	&&t^{\delta}
	\big\|\Im_{u_0}(u)(t)-\Im_{v_0}(v)(t)\big\|_{\dot{K}_{\kappa}^{\alpha,r}}
	\\
	&\le&
	t^{\delta}
	\big\|e^{t\Delta}(u_0-v_0)\big\|_{\dot{K}_{\kappa}^{\alpha,r}}
	\\
	&&
	+|a|t^{\delta}
	\int_0^t
	\big\|
	e^{(t-\tau)\Delta}
	\big(
	|\cdot|^{-\gamma}
	(|u(\tau)|^{\beta}u(\tau)-|v(\tau)|^{\beta}v(\tau))
	\big)
	\big\|_{\dot{K}_{\kappa}^{\alpha,r}}
	\,d\tau.
\end{eqnarray*}

By the embedding
\[
\dot{K}_{\kappa}^{\alpha,\frac{r}{\beta+1}}
\hookrightarrow
\dot{K}_{\kappa}^{\alpha,r},
\]
Lemma \ref{Bernstein-Herz-ine1}, and H\"{o}lder's inequality,
we obtain
\begin{eqnarray*}
	&&
	\big\|
	e^{(t-\tau)\Delta}
	\big(
	|\cdot|^{-\gamma}
	(|u(\tau)|^{\beta}u(\tau)-|v(\tau)|^{\beta}v(\tau))
	\big)
	\big\|_{\dot{K}_{\kappa}^{\alpha,r}}
	\\
	&\le&
	(t-\tau)^{-\frac12
		\left(
		\frac{n\beta}{\kappa}
		+\alpha\beta+\gamma
		\right)}
	\big\|
	|u(\tau)|^{\beta}u(\tau)
	-
	|v(\tau)|^{\beta}v(\tau)
	\big\|_{
		\dot{K}_{\frac{\kappa}{\beta+1}}^{
			\alpha(\beta+1),\frac{r}{\beta+1}}}
	\\
	&\le&
	c(\beta+1)
	(t-\tau)^{-\frac12
		\left(
		\frac{n\beta}{\kappa}
		+\alpha\beta+\gamma
		\right)}
	\Big(
	\|u(\tau)\|_{\dot{K}_{\kappa}^{\alpha,r}}^{\beta}
	+
	\|v(\tau)\|_{\dot{K}_{\kappa}^{\alpha,r}}^{\beta}
	\Big)
	\|u(\tau)-v(\tau)\|_{\dot{K}_{\kappa}^{\alpha,r}}
	\\
	&\le&
	2c(\beta+1)M^{\beta}
	(t-\tau)^{-\frac12
		\left(
		\frac{n\beta}{\kappa}
		+\alpha\beta+\gamma
		\right)}
	\tau^{-\delta(\beta+1)}
	d(u,v),
\end{eqnarray*}
where we used $\gamma\ge-\alpha\beta$, $\alpha>-\frac{n}{\kappa}$, and
\begin{equation}
\max\left(
\alpha-\gamma,
-\frac{n(\beta+1)}{\kappa}
\right)
\le
\alpha(\beta+1)
<
\min\left(
n-\frac{n(\beta+1)}{\kappa}-\gamma,
n-\frac{n(\beta+1)}{\kappa}
\right),
\label{alpha}
\end{equation}
which follows from \eqref{assumption1}.

Therefore,
\begin{eqnarray*}
	&&
	t^{\delta}
	\big\|
	\Im_{u_0}(u)(t)-\Im_{v_0}(v)(t)
	\big\|_{\dot{K}_{\kappa}^{\alpha,r}}
	\\
	&\le&
	t^{\delta}
	\big\|
	e^{t\Delta}(u_0-v_0)
	\big\|_{\dot{K}_{\kappa}^{\alpha,r}}
	\\
	&&
	+
	2|a|c(\beta+1)M^{\beta}
	t^{\delta}
	\int_0^t
	(t-\tau)^{-\frac12
		(\frac{n\beta}{\kappa}+\alpha\beta+\gamma)}
	\tau^{-\delta(\beta+1)}
	\,d\tau\,
	d(u,v)
	\\
	&=&
	t^{\delta}
	\big\|
	e^{t\Delta}(u_0-v_0)
	\big\|_{\dot{K}_{\kappa}^{\alpha,r}}
	\\
	&&
	+
	2|a|c_1(\beta+1)M^{\beta}
	t^{1+\delta
		-\frac12
		(\frac{n\beta}{\kappa}+\alpha\beta+\gamma)
		-\delta(\beta+1)}
	d(u,v),
\end{eqnarray*}
where
\[
c_1
=
c
\int_0^1
(1-\tau)^{
	-\frac12
	(\frac{n\beta}{\kappa}+\alpha\beta+\gamma)}
\tau^{-\delta(\beta+1)}
\,d\tau,
\]
and the constant $c$ is independent of $t$.

Observe that
\[
\frac1{q_c}
-\frac{2}{n(\beta+1)}
<
\frac1{\kappa},
\]
and
\[
\delta
=
\frac{n}{2q_c}
-
\frac{n}{2\kappa}
=
\frac{2-\alpha\beta-\gamma}{2\beta}
-
\frac{n}{2\kappa}.
\]
Hence,
\[
\delta(\beta+1)<1,
\]
and
\begin{equation}
1+\delta
-\frac12
\left(
\frac{n\beta}{\kappa}
+\alpha\beta+\gamma
\right)
-\delta(\beta+1)
=0.
\label{alpha2}
\end{equation}
Since $\delta>0$, we have
\[
\frac{n}{2\kappa}
<
\frac{2-\alpha\beta-\gamma}{2\beta},
\]
which implies
\[
\frac{n\beta}{2\kappa}
+\frac{\alpha\beta}{2}
+\frac{\gamma}{2}
<1.
\]
Consequently,
\begin{equation}
t^{\delta}
\big\|
\Im_{u_0}(u)(t)-\Im_{v_0}(v)(t)
\big\|_{\dot{K}_{\kappa}^{\alpha,r}}
\le
t^{\delta}
\big\|
e^{t\Delta}(u_0-v_0)
\big\|_{\dot{K}_{\kappa}^{\alpha,r}}
+
VM^{\beta}d(u,v),
\label{main-est1}
\end{equation}
where $V<\infty$.

Setting $v_0=0$ and $v=0$ in \eqref{main-est1}, we obtain
\begin{eqnarray*}
	\sup_{t>0}
	t^{\delta}
	\big\|
	\Im_{u_0}(u)(t)
	\big\|_{\dot{K}_{\kappa}^{\alpha,r}}
	&\le&
	\sup_{t>0}
	t^{\delta}
	\big\|
	e^{t\Delta}u_0
	\big\|_{\dot{K}_{\kappa}^{\alpha,r}}
	+
	VM^{\beta+1}
	\\
	&\le&
	\sigma
	+
	VM^{\beta+1}
	\\
	&\le&
	M.
\end{eqnarray*}
Hence,
\[
\Im_{u_0}(X_M)\subset X_M.
\]

Taking $u_0=v_0$ in \eqref{main-est1}, we deduce that
\[
d\bigl(
\Im_{u_0}(u),
\Im_{u_0}(v)
\bigr)
\le
VM^{\beta}d(u,v).
\]
Since $VM^{\beta}<1$, the mapping $\Im_{u_0}$ is a strict contraction on
$X_M$. Therefore, by the Banach fixed-point theorem,
$\Im_{u_0}$ admits a unique fixed point
$u\in X_M$, which is the unique mild solution to
\eqref{int-equa1}.

\textbf{Step 2.}
In this step, we prove assertions \emph{(i)}, \emph{(ii)}, and \emph{(iii)}.
From the integral equation, we have
\[
u(t)-e^{t\Delta}u_{0}
=
a\int_{0}^{t}
e^{(t-\tau)\Delta}
\bigl(|\cdot|^{-\gamma}|u(\tau)|^{\beta}u(\tau)\bigr)
\,d\tau.
\]
By Lemma \ref{Bernstein-Herz-ine1}, we obtain
\begin{eqnarray*}
	\|u(t)-e^{t\Delta}u_{0}\|_{\dot{K}_{s}^{\alpha,r}}
	&\le&
	|a|
	\int_{0}^{t}
	\Bigl\|
	e^{(t-\tau)\Delta}
	\bigl(
	|\cdot|^{-\gamma}|u(\tau)|^{\beta}u(\tau)
	\bigr)
	\Bigr\|_{\dot{K}_{s}^{\alpha,r}}
	\,d\tau
	\\
	&\le&
	c
	\int_{0}^{t}
	(t-\tau)^{-\frac12
		\left(
		\frac{n(\beta+1)}{\kappa}
		-\frac{n}{s}
		+\alpha\beta+\gamma
		\right)}
	\bigl\|
	|u(\tau)|^{\beta+1}
	\bigr\|_{
		\dot{K}_{\frac{\kappa}{\beta+1}}^{
			\alpha(\beta+1),\frac{r}{\beta+1}}}
	\,d\tau
	\\
	&\le&
	c
	\int_{0}^{t}
	(t-\tau)^{-\frac12
		\left(
		\frac{n(\beta+1)}{\kappa}
		-\frac{n}{s}
		+\alpha\beta+\gamma
		\right)}
	\|u(\tau)\|_{\dot{K}_{\kappa}^{\alpha,r}}^{\beta+1}
	\,d\tau
	\\
	&\le&
	cM^{\beta+1}
	\int_{0}^{t}
	(t-\tau)^{-\frac12
		\left(
		\frac{n(\beta+1)}{\kappa}
		-\frac{n}{s}
		+\alpha\beta+\gamma
		\right)}
	\tau^{-\delta(\beta+1)}
	\,d\tau,
\end{eqnarray*}
where we used $\gamma\ge-\alpha\beta$, $\alpha>-\frac ns$,
$s\ge\frac{\kappa}{\beta+1}$, and \eqref{alpha}.

Moreover,
\begin{eqnarray*}
	&&
	\int_{0}^{t}
	(t-\tau)^{-\frac12
		\left(
		\frac{n(\beta+1)}{\kappa}
		-\frac ns
		+\alpha\beta+\gamma
		\right)}
	\tau^{-\delta(\beta+1)}
	\,d\tau
	\\
	&=&
	t^{1
		-\frac12
		\left(
		\frac{n(\beta+1)}{\kappa}
		-\frac ns
		+\alpha\beta+\gamma
		\right)
		-\delta(\beta+1)}
	\\
	&&\times
	\int_{0}^{1}
	(1-\tau)^{-\frac12
		\left(
		\frac{n(\beta+1)}{\kappa}
		-\frac ns
		+\alpha\beta+\gamma
		\right)}
	\tau^{-\delta(\beta+1)}
	\,d\tau.
\end{eqnarray*}

Using \eqref{alpha2}, we obtain
\begin{eqnarray*}
	&&
	1
	-\frac12
	\left(
	\frac{n(\beta+1)}{\kappa}
	-\frac ns
	+\alpha\beta+\gamma
	\right)
	-\delta(\beta+1)
	\\
	&=&
	\frac{n}{2s}
	-\delta
	-\frac{n}{2\kappa}
	\\
	&=&
	\frac{n}{2s}
	-\frac{n}{2q_c}
	+\frac{n}{2\kappa}
	-\frac{n}{2\kappa}
	\\
	&=&
	\frac{n}{2s}
	-\frac{n}{2q_c}
	\\
	&=&
	\frac{n}{2s}
	-
	\frac{2-\alpha\beta-\gamma}{2\beta}.
\end{eqnarray*}
Therefore,
\begin{equation}
\|u(t)-e^{t\Delta}u_{0}\|_{\dot{K}_{s}^{\alpha,r}}
\le
cM^{\beta+1}
t^{\frac{n}{2s}-\frac{2-\alpha\beta-\gamma}{2\beta}}.
\label{main-est2}
\end{equation}

Since
\[
\frac{n}{2s}
-
\frac{2-\alpha\beta-\gamma}{2\beta}
>0,
\]
it follows that
\[
t^{\frac{n}{2s}-\frac{2-\alpha\beta-\gamma}{2\beta}}
\longrightarrow0
\qquad\text{as }t\searrow0.
\]
Hence,
\[
u(t)-e^{t\Delta}u_{0}
\longrightarrow0
\quad\text{in }\dot{K}_{s}^{\alpha,r},
\]
which proves assertions \emph{(i)} and \emph{(iii)}. Finally, assertion
\emph{(ii)} follows directly from \eqref{main-est2} by taking
$s=q_c$.

\textbf{Step 3.}
In this step, we prove the stronger decay estimate \eqref{estu-v2}. By the
estimate obtained in Step~1, we have
\begin{equation*}
\|u(t)-v(t)\|_{\dot{K}_{\kappa}^{\alpha,r}}
\end{equation*}
is bounded by
\begin{equation*}
\|e^{t\Delta}(u_{0}-v_{0})\|_{\dot{K}_{\kappa}^{\alpha,r}}
+c|a|(\beta+1)M^{\beta}
\int_{0}^{t}
(t-\tau)^{-\frac12
	\left(
	\frac{n\beta}{\kappa}
	+\alpha\beta+\gamma
	\right)}
\tau^{-\delta\beta}
\|u(\tau)-v(\tau)\|_{\dot{K}_{\kappa}^{\alpha,r}}
\,d\tau .
\end{equation*}

Let $\eta>0$ be such that
\[
\delta(\beta+1)+\eta<1.
\]
Then, for every $T>0$,
\begin{eqnarray*}
	&&
	t^{\delta+\eta}
	\int_{0}^{t}
	(t-\tau)^{-\frac12
		\left(
		\frac{n\beta}{\kappa}
		+\alpha\beta+\gamma
		\right)}
	\tau^{-\delta\beta}
	\|u(\tau)-v(\tau)\|_{\dot{K}_{\kappa}^{\alpha,r}}
	\,d\tau
	\\
	&\lesssim&
	t^{1
		-\frac12
		\left(
		\frac{n\beta}{\kappa}
		+\alpha\beta+\gamma
		\right)
		-\delta\beta}
	\sup_{0<\tau\le T}
	\tau^{\delta+\eta}
	\|u(\tau)-v(\tau)\|_{\dot{K}_{\kappa}^{\alpha,r}}
	\\
	&&\times
	\int_{0}^{1}
	(1-h)^{-\frac12
		\left(
		\frac{n\beta}{\kappa}
		+\alpha\beta+\gamma
		\right)}
	h^{-\delta(\beta+1)-\eta}
	\,dh
	\\
	&=&
	C
	\sup_{0<\tau\le T}
	\tau^{\delta+\eta}
	\|u(\tau)-v(\tau)\|_{\dot{K}_{\kappa}^{\alpha,r}},
\end{eqnarray*}
where $C>0$ is independent of $T$. Here we used the facts that
$\delta(\beta+1)+\eta<1$, $\kappa>q_c$, and
\[
1-\frac12
\left(
\frac{n\beta}{\kappa}
+\alpha\beta+\gamma
\right)
-\delta\beta
=0.
\]

Consequently,
\begin{align*}
\sup_{0<t\le T}
t^{\delta+\eta}
\|u(t)-v(t)\|_{\dot{K}_{\kappa}^{\alpha,r}}
&\le
\sup_{0<t\le T}
t^{\delta+\eta}
\|e^{t\Delta}(u_{0}-v_{0})\|_{\dot{K}_{\kappa}^{\alpha,r}}
\\
&\quad
+C|a|(\beta+1)M^{\beta}
\sup_{0<t\le T}
t^{\delta+\eta}
\|u(t)-v(t)\|_{\dot{K}_{\kappa}^{\alpha,r}}.
\end{align*}

Since $C|a|(\beta+1)M^{\beta}<1$, the last term can be absorbed into the
left-hand side. Hence,
\[
\sup_{0<t\le T}
t^{\delta+\eta}
\|u(t)-v(t)\|_{\dot{K}_{\kappa}^{\alpha,r}}
\lesssim
\sup_{0<t\le T}
t^{\delta+\eta}
\|e^{t\Delta}(u_{0}-v_{0})\|_{\dot{K}_{\kappa}^{\alpha,r}}.
\]
This proves the stronger decay estimate \eqref{estu-v2}.

\textbf{Step 4.} We prove (iv). By assumption,
\[
\sup_{t>0}t^{\frac{2-\alpha\beta-\gamma}{2\beta}-\frac{n}{2\kappa}}
\|u(t)\|_{\dot{K}_{\kappa}^{\alpha,r}}<\infty.
\]

We apply Lemma \ref{keylemma1} iteratively, following the argument of
\cite{STW01}. Set $s_{0}=\kappa$, and choose $s_{1}$ such that
\[
\beta+1<s_{0}<s_{1}\leq\infty,\qquad
\alpha>-\frac{n}{s_{1}},\qquad
-\beta\alpha\leq\gamma<n,
\]
together with
\[
\frac{n}{s_{1}}
<
\min\left(
\frac{n-\gamma}{\beta+1}-\alpha,
\frac{n}{\beta+1}-\alpha
\right)
\]
and
\begin{equation}
\frac{\beta+1}{s_{0}}-\frac{1}{s_{1}}
<
\frac{2-\alpha\beta-\gamma}{n}.
\label{main1}
\end{equation}

Then Lemma \ref{keylemma1} yields
\[
\sup_{t>0}
t^{\frac{2-\alpha\beta-\gamma}{2\beta}-\frac{n}{2s_{1}}}
\|u(t)\|_{\dot{K}_{s_{1}}^{\alpha,r}}
<\infty.
\]

Repeating the argument, we choose $s_{2}$ satisfying
\[
\beta+1<s_{1}<s_{2}\leq\infty,\qquad
\alpha>-\frac{n}{s_{2}},\qquad
-\beta\alpha\leq\gamma<n,
\]
\begin{equation}
\frac{\beta+1}{s_{1}}-\frac{1}{s_{2}}
<
\frac{2-\alpha\beta-\gamma}{n},
\label{main2}
\end{equation}
and
\[
\frac{n}{s_{2}}
<
\min\left(
\frac{n-\gamma}{\beta+1}-\alpha,
\frac{n}{\beta+1}-\alpha
\right).
\]

Again, Lemma \ref{keylemma1} implies
\[
\sup_{t>0}
t^{\frac{2-\alpha\beta-\gamma}{2\beta}-\frac{n}{2s_{2}}}
\|u(t)\|_{\dot{K}_{s_{2}}^{\alpha,r}}
<\infty.
\]

Combining \eqref{main1} and \eqref{main2}, we obtain
\[
\begin{aligned}
\frac1{s_2}
&>
\frac{\beta+1}{s_1}
-
\frac{2-\alpha\beta-\gamma}{n}
\\
&>
\frac{(\beta+1)^2}{s_0}
-
\frac{(\beta+2)(2-\alpha\beta-\gamma)}{n}.
\end{aligned}
\]

Proceeding inductively, we construct a sequence $\{s_i\}_{i\ge1}$ such that
\[
\beta+1<s_{i-1}<s_i\le\infty,\qquad
\alpha>-\frac{n}{s_i},\qquad
-\beta\alpha\le\gamma<n,
\]
\[
\frac{\beta+1}{s_{i-1}}-\frac1{s_i}
<
\frac{2-\alpha\beta-\gamma}{n},
\]
and
\[
\frac{n}{s_i}
<
\min\left(
\frac{n-\gamma}{\beta+1}-\alpha,
\frac{n}{\beta+1}-\alpha
\right).
\]

Moreover,
\[
\begin{aligned}
\frac1{s_i}
&>
\frac{(\beta+1)^i}{s_0}
-
\sum_{j=0}^{i-1}
\frac{(\beta+1)^j(2-\alpha\beta-\gamma)}{n}
\\
&=
\frac{(\beta+1)^i}{s_0}
+
\bigl(1-(\beta+1)^i\bigr)
\frac{2-\alpha\beta-\gamma}{\beta n}.
\end{aligned}
\]

Consequently,
\[
\sup_{t>0}
t^{\frac{2-\alpha\beta-\gamma}{2\beta}-\frac{n}{2s_i}}
\|u(t)\|_{\dot{K}_{s_i}^{\alpha,r}}
<\infty.
\]

Let $i_{0}\in\mathbb N$ be such that
\[
\kappa\le q\le s_{i_{0}}
\quad\text{and}\quad
\frac1{s_{i_{0}}}>-\frac{\alpha}{n}.
\]
Such a choice is possible since
\[
\frac1{q_c}-\frac1q
\ge
\frac1{q_c}-\frac1\kappa.
\]

Choose $\theta\in[0,1]$ satisfying
\[
\frac1q
=
\frac{1-\theta}{\kappa}
+
\frac{\theta}{s_{i_{0}}}.
\]

Applying the interpolation inequality \eqref{Interpolation}, we obtain
\[
\begin{aligned}
&t^{\frac{2-\alpha\beta-\gamma}{2\beta}-\frac{n}{2q}}
\|u(t)\|_{\dot{K}_q^{\alpha,r}}
\\
&\le
\left(
t^{\frac{2-\alpha\beta-\gamma}{2\beta}-\frac{n}{2\kappa}}
\|u(t)\|_{\dot{K}_{\kappa}^{\alpha,r}}
\right)^{1-\theta}
\left(
t^{\frac{2-\alpha\beta-\gamma}{2\beta}-\frac{n}{2s_{i_0}}}
\|u(t)\|_{\dot{K}_{s_{i_0}}^{\alpha,r}}
\right)^\theta
\\
&\le
M^{1-\theta}
\left(
t^{\frac{2-\alpha\beta-\gamma}{2\beta}-\frac{n}{2s_{i_0}}}
\|u(t)\|_{\dot{K}_{s_{i_0}}^{\alpha,r}}
\right)^\theta
<\infty.
\end{aligned}
\]

Thus, it remains only to prove (iv) in the case $q=\infty$.

Observe that
\[
\lim_{i\to\infty}
\left(
\frac1{s_0}
+
\left(
\frac1{(\beta+1)^i}-1
\right)
\frac{2-\alpha\beta-\gamma}{\beta n}
\right)
=
\frac1{s_0}
-
\frac{2-\alpha\beta-\gamma}{\beta n}
=
-\frac{2\delta}{n}
<0.
\]

Hence, there exists $i_{1}\in\mathbb N$ such that
\[
\frac{(\beta+1)^{i_1}}{s_0}
+
\left(1-(\beta+1)^{i_1}\right)
\frac{2-\alpha\beta-\gamma}{\beta n}
<
0
=
\frac1\infty.
\]

We therefore set $s_{i_1}=\infty$. Applying Lemma \ref{keylemma1} once more gives
\[
\sup_{t>0}
t^{\frac{2-\alpha\beta-\gamma}{2\beta}}
\|u(t)\|_{\dot{K}_{\infty}^{\alpha,r}}
<\infty,
\]
which is precisely the desired estimate.

\textbf{Step 5.} The continuous dependence estimate \eqref{estu-v1} for
$q=\kappa$ follows directly from \eqref{main-est1} by taking
\[
\Im_{u_{0}}(u)=u
\quad\text{and}\quad
\Im_{v_{0}}(v)=v.
\]
The estimate \eqref{estu-v1} for all $q\in[\kappa,\infty]$, with the additional
assumption $\alpha\geq0$ when $q=\infty$, can then be established by an
iterative argument similar to that used in the proof of (iv). More precisely,
one proves an analogue of Lemma \ref{keylemma1} with $u$ replaced by $u-v$.
This completes the proof of the theorem.
\end{proof}

\begin{remark}
	If $\alpha=0$ and $r=\kappa$, then Theorem~\ref{result5} reduces to
	\cite[Theorem~4.1]{BTW17}. Moreover, the assumptions on $\gamma$ are implied
	by the conditions $q_{c}>1$, $\alpha-\gamma\leq\alpha(\beta+1)$, and
	\eqref{assumption1}. Finally, the assumption on $\alpha$ follows from the
	assumptions on $\gamma$.
	\end{remark}

\begin{theorem}
\label{result5 copy(1)} Under the hypotheses of Theorem~\ref{result5}, the
following assertions hold.

\begin{enumerate}
\item[(i)] 
\begin{equation*}
u(t)-e^{t\Delta }u_{0}\in C([0,\infty );\dot{K}_{s}^{\alpha ,r}),
\end{equation*}%
provided that 
\begin{equation*}
\max \left( \frac{1}{q_{c}},-\frac{\alpha }{n},\frac{\beta +1}{\kappa }%
\right) <\frac{1}{s}<\frac{\beta +1}{\kappa }+\frac{\gamma +\alpha \beta }{n}%
.
\end{equation*}

\item[(ii)] 
\begin{equation*}
u(t)-e^{t\Delta }u_{0}\in C((0,\infty );\dot{K}_{s}^{\alpha ,r}),
\end{equation*}%
provided that 
\begin{equation*}
\max \left( \frac{1}{q_{c}},-\frac{\alpha }{n},\frac{\beta +1}{\kappa }%
\right) \leq \frac{1}{s}<\frac{\beta +1}{\kappa }+\frac{\gamma +\alpha \beta 
}{n},
\end{equation*}%
with 
\begin{equation*}
\frac{1}{s}\neq -\frac{\alpha }{n},\qquad \frac{1}{s}\neq \frac{\beta +1}{%
\kappa }.
\end{equation*}
\end{enumerate}
\end{theorem}

\begin{proof}
The proof follows the same argument as that of Theorem~\ref{result5}, with
Lemma~\ref{Bernstein-Herz-ine2} replacing Lemma~\ref{Bernstein-Herz-ine1}.
\end{proof}

\begin{remark}
	Additional results on the Hardy--H\'{e}non parabolic equation in Herz spaces, such as the asymptotic behavior of global mild solutions, are postponed to future work.
\end{remark}


\end{document}